\documentclass{amsart}%
\usepackage{amsmath}
\usepackage{amsfonts}
\usepackage{amssymb}
\usepackage{graphicx}%
\newtheorem{lemma}{Lemma}[section]
\newtheorem{theorem}[lemma]{Theorem}
\newtheorem{remark}[lemma]{Remark}

\newtheorem{coro}[lemma]{Corollary}
\newtheorem{definition}[lemma]{Definition}
\newtheorem{example}[lemma]{Example}

\title[Approximation of Attractors \ldots]{Approximation of Attractors of Nonautonomous Lattice Dynamical Systems}

\author{David Cheban}
\address[D. Cheban]{State University of Moldova\\
Vladimir Andrunachievici Instiy=tute of Mathematics and Computer Science\\
Laboratory of Differential Equations\\str. Academiei 5\\
MD--2028 Chi\c{s}in\u{a}u, Moldova} \email[D.
Cheban]{david.ceban@usm.md, davidcheban@yahoo.com}
\author{Andrei Sultan}
\address[A. Sultan]{%
State University of Moldova\\Vladimir Andrunachievici Institute of
Mathematics and Computer Science\\ Laboratory of Differential Equations\\str. Academiei 5\\
MD--2028 Chi\c{s}in\u{a}u, Moldova} \email[A.
Sultan]{andrew15sultan@gmail.com}

\date{\today}

\subjclass{34D05, 34D45, 34G20, 37B55}

\begin{document}

\begin{abstract}
The aim of this paper is to study the finite-dimensional
approximations of the nonautonomous lattice dynamical systems of
the form $u_{i}'=\nu (u_{i-1}-2u_i+u_{i+1})-\lambda
u_{i}+F(u_i)+f_{i}(t)\ (i\in \mathbb Z)\ (*)$. We show that the
finite-dimensional approximations for (*) are uniformly
dissipative. The upper semi-continuous convergence of the
attractors of the finite-dimensional approximations is
established.
\end{abstract}

\maketitle

\section{Introduction}\label{Sec1}

Denote by $\mathbb R :=(-\infty,\infty)$, $\mathbb Z :=\{0,\pm
1,\pm 2,\ldots\}$ and $\ell_{2}$ the Hilbert space of all
two-sided sequences $\xi =(\xi_{i})_{i\in \mathbb Z}$ ($\xi_{i}\in
\mathbb R$) with
\begin{equation}\label{eqI_1}
\sum\limits_{i\in \mathbb Z}|\xi_{i}|^{2}<+\infty \nonumber
\end{equation}
and equipped with the scalar product
\begin{equation}\label{eqI_2}
\langle \xi,\eta\rangle :=\sum\limits_{i\in \mathbb
Z}\xi_{i}\eta_{i} .\nonumber
\end{equation}
Let $(\mathfrak B, \|\cdot \|_{\mathfrak B})\footnote{In what
follows, in the notation $\|\cdot \|_{\mathfrak B}$, we will omit
the index $\mathfrak B$ if this does not lead to a
misunderstanding.}$ be a Banach space with the norm
$\|\cdot\|_{\mathfrak B}$, $C(\mathbb R,\mathfrak B)$ be the space
of all continuous functions $f:\mathbb R\to \mathfrak B$ equipped
with the distance
\begin{equation}\label{eqI_3}
d(f_1,f_2):=\sup\limits_{L>0}\min\{\max\limits_{|t|\le
L}\|f_{1}(t)-f_{2}(t)\|,L^{-1}\}.
\end{equation}
The metric space $(C(\mathbb R,\mathfrak B),d)$ is complete and
the distance $d$, defined by (\ref{eqI_3}), generates on the space
$C(\mathbb R,\mathfrak B)$ the compact-open topology.

Let $h\in \mathbb R$, $f\in C(\mathbb R,\mathfrak B)$,
$f^{h}(t):=f(t+h)$ for all $t\in \mathbb R$ and $\sigma :\mathbb
R\times C(\mathbb R,\mathfrak B)\to C(\mathbb R,\mathfrak B)$ be a
mapping defined by $\sigma(h,f):=f^{h}$ for every $(h,f)\in \mathbb
R\times C(\mathbb R,\mathfrak B)$. Then \cite[Ch.I]{Che_2015} the
triplet $(C(\mathbb R,\mathfrak B),\mathbb R,\sigma)$ is a shift
dynamical system (or Bebutov's dynamical system) on he space
$C(\mathbb R,\mathfrak B)$. By $H(f)$ the closure in the space
$C(\mathbb R,\mathfrak B)$ of $\{f^{h}|\ h\in \mathbb R\}$ is
denoted.

Denote by $\mathfrak F$ a subset of the space $C(\mathbb
R,\ell_{2})$.

\begin{definition}\label{defUSC1} A family of function $\mathfrak F$ from
$C(\mathbb T,\ell_{2})$ is said to be uniformly equi-continuous if
for arbitrary positive number $\varepsilon$ there exists a
positive number $\delta =\delta(\varepsilon)$ such that
$|t_1-t_2|<\delta$ ($t_1,t_2\in \mathbb T$) implies
$\|\varphi(t_1)-\varphi(t_2)\|<\varepsilon$ for all $\varphi \in
\mathfrak F$.
\end{definition}

\begin{lemma}\label{lAPF02} \cite{Sel_1971},\cite[Ch.III]{scher72},\cite{sib} A set $\mathfrak F$ is precompact in the space
$C(\mathbb T,\ell_{2})$ if nd only if the following conditions are
fulfilled:
\begin{enumerate}
\item there exists a compact subset $Q$ from $\ell_{2}$ such that
$\varphi(\mathbb T)\subseteq Q$ for every $\varphi \in \mathfrak
F$; \item the family of functions $\mathfrak F$ from $C(\mathbb
T,\ell_{2})$ is uniformly equi-continuous.
\end{enumerate}
\end{lemma}

In this paper we study the finite-dimensional approximations and
upper semi-continuous convergence of compact global attractors of
the systems
\begin{equation}\label{eqI1}
u_{i}'=\nu (u_{i-1}-2u_i+u_{i+1})-\lambda u_{i}+F(u_i)+f_{i}(t)\
(i\in \mathbb Z),
\end{equation}
where $\lambda >0$, $F\in C(\mathbb R, \mathbb R)$ and $f\in
C(\mathbb R,\ell_{2})$ ($f(t):=(f_{i}(t))_{i\in \mathbb Z}$ for
all $t\in \mathbb R$).

The system (\ref{eqI1}) can be considered as a discrete (see, for
example, \cite{BLW_2001}, \cite{HK_2023} and the bibliography
therein) analogue of a reaction-diffusion equation in $\mathbb R$:
\begin{equation}\label{eqI1.1}
 \frac{\partial{u}}{\partial{t}} = D\frac{\partial^{2}{u}}{\partial^{2}{x}}-\lambda u + F(u) +
 f(t,x),\nonumber
\end{equation}
where grid points are spaced $h$ distance apart and $\nu =
D/h^{2}$.

This study continues the first author's works devoted to the study
of compact global attractors of non-autonomous dynamical systems
\cite{Che_2015} and compact global attractors of lattice dynamical
systems \cite{BLW_2001} (autonomous systems), compact pullback
attractors \cite{HK_2023} (for non-autonomous systems) and compact
(forward) global attractors and almost periodic solutions for
nonautonomous lattice dynamical systems \cite{CS_2025,CS_2026}.

In this paper we study the problem of approximation of compact
global attractors of nonautonomous lattice dynamical systems.
Following \cite{BLW_2001} (see also \cite{HK_2023}) we consider
finite-dimensional systems which are approximations to the
original infinite-dimensional equation (\ref{eqI1}). The existence
of compact global attractors for these approximative systems is
established and also we prove that the global attractors for the
approximative systems converge to the global attractor of
(\ref{eqI1}) with respect to (w.r.t.) Hausdorff's semi-distance (upper
semi-continuity).

The paper is organized as follows.

In the second section we collect some notion and facts from the
theory of dynamical systems (both autonomous and nonautonomous)
which we will use in this paper.

The third section is dedicated to the construction of the
finite-dimensional approximations for the lattice dynamical
systems (\ref{eqI1}).

In the fourth section we establish some properties of the shift
dynamical systems related with the nonautonomous lattice dynamical
systems and their finite-dimensional approximations.

The fifth section is dedicated to the construction of the cocycles
generated by nonautonomous lattice dynamical systems (\ref{eqI1})
and their finite-dimensional approximations.

In the sixth section we establish the uniform dissipativity and
asymptotically compactness of the cocycles generated by
nonautonomous lattice dynamical systems (\ref{eqI1}).

The seventh section is dedicated to the study the problem of upper
semi-continuous convergence of the compact global attractors for
finite-dimensional approximations for (\ref{eqI1}).

\section{Preliminary}\label{Sec2}

\begin{definition}\label{def1.0.1}
Triplet $(X,\mathbb{T},\pi)$, where $\pi:\mathbb{T}\times X\to X$
is a continuous mapping satisfying the following conditions:
$\pi(0,x)=x$ and $\pi(s,\pi(t,x))=\pi(s+t,x)$ is called a
dynamical system.
\end{definition}

If $\mathbb{T}=\mathbb{R}$ $(\mathbb{R_{+}})$, then
$(X,\mathbb{T},\pi)$ is called a group (semi-group) dynamical
system.

\begin{definition}\label{def1.2.2*}
The system $(X,\mathbb{T},\pi)$ is called compact
dissipative\index{compact dissipative} if there exist a compact
subset $K\subseteq X$ such that
\begin{equation}\label{eq1.2.1}
\lim \limits_{t\to+\infty}\beta(\pi(t,M),K)=0;
\end{equation}
for every compact subset $M$ from $X$.
\end{definition}

Let $M$ be a subset of $X$ and denote by
\begin{equation}\label{eq1.2.3}
\omega(M):=\bigcap\limits_{t\ge0}\overline{\bigcup
\limits_{\tau\ge t}\pi(\tau ,M)}.\nonumber
\end{equation}

Let $(X,\mathbb{T},\pi)$ be compact dissipative and $K$ is a
nonempty compact set figuring in the relation (\ref{eq1.2.1}).
Then $\omega(K)\subseteq K$ and, consequently
\cite[Ch.I]{Che_2015},
\begin{equation}\label{eq1.2.4}
J=\cap\{\pi^{t}M\ |\ t\in \mathbb{T}\}.
\end{equation}
The set $J$ does not depend on the choice of the set $K$ (see
\cite[Ch.I]{Che_2015}).

\begin{definition}\label{def1.2.3*}
The set $J$ defined by the equality (\ref{eq1.2.4}), according to
\cite{Zhi72} we will call the center of Levinson of the compact
dissipative dynamical system $(X,\mathbb{T},\pi)$.
\end{definition}

\begin{theorem}\label{t1.2.4}\cite[Ch.I]{Che_2015}
Let $(X,\mathbb{T},\pi)$ be compact dissipative system and let $J$
be its center of Levinson. Then:
\begin{enumerate}
\item $J$ is compact invariant set; \item $J$ is orbitally stable;
\item $J$ is the attractor of the family of all compacts of $X$;
\item $J$ is the maximal compact invariant set in
$(X,\mathbb{T},\pi)$.
\end{enumerate}
\end{theorem}

\begin{definition}\label{def1.0.17}
Let $(X,\mathbb{T}_1,\pi)$ and $(Y,\mathbb{T}_2,\sigma)$ \
($\mathbb{R_{+}}\subseteq\mathbb{T}_1\subseteq\mathbb{T}_2
\subseteq\mathbb{R}$) be two dynamical systems. A mapping $h:X\to
Y$ is called a homomorphism of dynamical system
$(X,\mathbb{T}_1,\pi)$ on $(Y,\mathbb{T}_2,\sigma)$, if $h$ is
continuous and $h(\pi(x,t))=\sigma(h(x),t)$ ( for all
$t\in\mathbb{T}_1,\ x\in X$).
\end{definition}

\begin{definition}\label{def1.0.18}
The triplet $\langle
(X,\mathbb{T}_1,\pi),\,(Y,\mathbb{T}_2,\sigma), \,h\rangle $,
where $h$ is a homomorphism from $(X,\mathbb{T}_1,\pi)$ on
$(Y,\mathbb{T}_2,\sigma)$, is called a non-autonomous dynamical
system.
\end{definition}

\begin{definition}\label{def1.0.19}
The triplet $\langle W, \varphi, (Y,\mathbb{T}_2,\sigma)\rangle $
(or shortly $\varphi$), where $(Y,\mathbb{T}_2,\sigma)$ is a
dynamical system on $Y$,  $W$ is a complete metric space and
$\varphi$ is a continuous mapping from $\mathbb{T}_1\times W\times
Y$ in $W$, possessing the following conditions:
\begin{enumerate}
\item[a.] $\varphi(0,u,y)=u$ (for all $u\in W, y\in Y)$; \item[b.]
$\varphi(t+\tau,u,y)= \varphi(\tau,\varphi(t,u,y),\sigma(t,y))$
(for all $t,\tau\in\mathbb{T}_1,\, u\in W, y\in Y),$
\end{enumerate}
is called \cite{Arn,Sel_1971} a cocycle on
$(Y,\mathbb{T}_2,\sigma)$ with the fibre $W$.
\end{definition}

\begin{definition}\label{def1.0.20}
Let $X:= W\times Y$ and we define a mapping $\pi: X\times
\mathbb{T}_1\to X$ as following:
$\pi((u,y),t):=(\varphi(t,u,y),\sigma(t,y))$ (i.e.
$\pi=(\varphi,\sigma)$) for all $(t,u,y)\in \mathbb T_{1}\times
W\times Y$. Then it easy to see that $(X,\mathbb{T}_1,\pi)$ is a
dynamical system on $X$ which is called a skew-product dynamical
system \cite{Arn,Sel_1971} and $h=pr_2:X\to Y$ is a homomorphism
from $(X,\mathbb{T}_1,\pi)$ on $(Y,\mathbb{T}_2,\sigma)$ and,
consequently, $\langle (X,\mathbb{T}_1,\pi),\,
(Y,\mathbb{T}_2,\sigma), h\rangle $ is a non-autonomous dynamical
system.
\end{definition}

Thus, if we have a cocycle $\langle W, \varphi, (Y,\mathbb{T}_2,
\sigma)\rangle $ on dynamical system $(Y,\mathbb{T}_2,\sigma)$
with the fibre $W$, then it generates a non-autonomous dynamical
system $\langle (X,\mathbb{T}_1,\pi),$\ $(Y,\mathbb{T}_2,\sigma),
h\rangle $ ($X:= W\times Y$), which is called a non-autonomous
dynamical system, generated by cocycle\index{non-autonomous
dynamical system, generated by cocycle} $\langle W, \varphi,
(Y,\mathbb{T}_2,\sigma)\rangle $ on $(Y,\mathbb{T}_2,\sigma)$.

Non-autonomous dynamical systems (cocycles) play a very important
role in the study of non-autonomous evolutionary differential
equations. Under appropriate assumptions every non-autonomous
differential equation generates some cocycle (non-autonomous
dynamical system). Below we give an example of this type.

Let $(X,\mathbb{T},\pi)$ be a dynamical system on $X,\,Y$ be a
complete pseudo metric space and $\mathcal{P}$ be a family of
pseudo metrics on $Y$. Denote by $C(X,Y)$ the family of all
continuous functions $f:X\to Y$ equipped with the compact-open
topology. This topology is given by the following family of pseudo
metrics $\{d_{K}^p\}\ (p\in\mathcal{P},\;\,K\in C(X)),$ where $$
d_{K}^p(f,g):=\sup\limits_{x\in K}p(f(x),g(x))$$ and $C(X)$ a
family of all compact subsets of $X$. We define for all
$\tau\in\mathbb{T}$ the mapping $\sigma_{\tau}:C(X,Y)\to C(X,Y)$
by the following equality: $(\sigma_{\tau}f)(x):=f(\pi(\tau,x))$ \
$(x\in X)$. We note that the family of
mappings$\{\sigma_{\tau}:\tau\in\mathbb{T}\}$ possesses the
following properties:
\begin{enumerate}
\item[a.] $\sigma_0=id_{C(X,Y)}$; \ \ \item[b.]
$\sigma_{\tau_1}\circ\sigma_{\tau_2}=\sigma_{\tau_1+\tau_2}$ for
all $\tau_1, \tau _2 \in \mathbb T$; \ \ \item[c.] $\sigma_{\tau}$
is continuous for all $\tau \in \mathbb T$.
\end{enumerate}

\begin{lemma}\label{l1.0.5} \cite[Ch.I]{Che_2015}
The mapping $\sigma:\mathbb{T}\times C(X,Y)\to C(X,Y),$ defined by
the equality $\sigma(\tau,f):=\sigma_{\tau}f$ \ $(f\in
C(X,Y),\;\tau\in\mathbb{T}),$ is continuous.
\end{lemma}

\begin{example}\label{ex1.0.9}
{\rm Let $X:=\mathbb{T}\times W$, where $W$ some metric space and
by $(X,\mathbb{T},\pi)$ we denote a dynamical system on $X$
defined in the following way: $\pi(t,(s,w)):=(s+t,w)$. Using the
general method proposed above we can define on $C(\mathbb{T}\times
W,Y)$ a dynamical system of translations $(C(\mathbb{T}\times
W,Y),\mathbb{T},\sigma)$.}
\end{example}

\begin{example}\label{ex1.1.13}
{\rm Let us consider a differential equation
\begin{equation}
u'=f(t,u),\label{eq1.0.6}
\end{equation}
where  $f\in C(\mathbb{R}\times \mathfrak B,\mathfrak B)$,
$\mathfrak B$ is a Banach space with the norm $\|\cdot\|$. Along
with equation (\ref{eq1.0.6}) we consider its $H$-class
\cite{Bro79,Dem67,Lev-Zhi,scher72,scher85}, i.e., the family of
equations
\begin{equation}
v'=g(t,v),\label{eq1.0.7}
\end{equation}
where $g\in H(f)=\overline{\{f_{\tau}:\tau\in \mathbb{R}\}}$,
$f_{\tau}(t,u)=f(t+\tau,u)$ for all $(t,u)\in \mathbb{R}\times
\mathfrak B$ and by bar we denote the closure in
$C(\mathbb{R}\times \mathfrak B,\mathfrak B)$. We will suppose
also that the function $f$ is regular, i.e., for every equation
(\ref{eq1.0.7}) the conditions of existence, uniqueness and
extendability on $\mathbb{R}_{+}$ are fulfilled. Denote by
$\varphi(\cdot,v,g)$ the solution of the equation (\ref{eq1.0.7}),
passing through the point $v\in \mathfrak B$ at the initial moment
$t=0$. Then it is well defined a mapping
$\varphi:\mathbb{R}_{+}\times \mathfrak B\times H(f)\to \mathfrak
B$, verifying the following conditions (see, for example,
\cite{Bro79,Sel_1971}):
\begin{enumerate}
\item[$1)$] $\varphi(0,v,g)=v$ for all $v\in \mathfrak B$ and
$g\in H(f)$; \item[$2)$]
$\varphi(t,\varphi(\tau,v,g),g_{\tau})=\varphi(t+\tau,v,g)$ for
every $ v\in \mathfrak B$, $g\in H(f)$ and $t,\tau \in
\mathbb{R}_{+}$; \item[$3)$] the mapping
$\varphi:\mathbb{R}_{+}\times \mathfrak B\times H(f)\to \mathfrak
B$ is continuous.
\end{enumerate}

Denote by $Y:=H(f)$ and $(Y,\mathbb{R},\sigma)$ the dynamical
system of translations on $Y$, induced by dynamical system of
translations $(C(\mathbb{R}\times \mathfrak B,\mathfrak
B),\mathbb{R},\sigma)$. The triplet $\langle \mathfrak B,\varphi,
(Y,\mathbb{R},\sigma)\rangle $ is a cocycle on
$(Y,\mathbb{R},\sigma)$ with the fibre $\mathfrak B$. Thus the
equation (\ref{eq1.0.6}) generates a cocycle $\langle \mathfrak
B,\varphi, (Y,\mathbb{R},\sigma)\rangle $ and a non-autonomous
dynamical system $\langle (X,\mathbb{R}_{+},\pi),\,
(Y,\mathbb{R},\sigma), h\rangle $, where $X:= \mathfrak B\times
Y$, $\pi:=(\varphi,\sigma)$ and $h:=pr_2:X\to Y$.}
\end{example}

\begin{definition}\label{def2.7.3}
A cocycle $\varphi $ over $(Y,\mathbb T,\sigma)$ with the fiber
$W$ is said to be compactly dissipative (respectively, uniformly
compact dissipative) if there exits a nonempty compact $K
\subseteq W$ such that
\begin{equation}\label{eqGA2.7.8}
\lim_{t \to + \infty} \beta (U(t,y)M,K) =0
\end{equation}
for any  $M \in C(W)$ and $y\in Y$ (respectively, uniformly with
respect to $y\in Y$).
\end{definition}

\begin{theorem}\label{thGA1}\cite{Che_1997},\ \cite[Ch.II]{Che_2015} Let $Y$ be a compact metric space,
then the following statements are equivalent:
\begin{enumerate}
\item the cocycle $\langle W,\varphi,(Y,\mathbb T,\sigma)\rangle$
is uniformly compactly dissipative; \item the skew-product
dynamical system $(X,\mathbb T,\pi)$ ($X:=W\times Y, \pi
=(\varphi,\sigma)$) is compact dissipative.
\end{enumerate}
\end{theorem}

\begin{definition}\label{defGA1} A non-autonomous set $I=\{I_{y}|\ y\in
Y\}$ is said to be a compact global attractor for the cocycle
$\langle W,\varphi, (Y,\mathbb T,\sigma)\rangle$ if it possesses
the following properties:
\begin{enumerate}
\item the set $\mathcal I :=\bigcup \{I_{y}|\ y\in Y\}$ is
pre-compact; \item $\{I_{y}|\ y\in Y\}$ is invariant, i.e.,
$\varphi(t,I_{y},y)=I_{\sigma(t,y)}$ for any $(t,y)\in \mathbb T
\times Y$; \item
\begin{equation}\label{eqGA9}
\lim\limits_{t\to +\infty}\sup\limits_{y\in
Y}\beta(\varphi(t,M,y),\mathcal I)=0\nonumber
\end{equation}
for any $M\in C(W)$, where $\mathcal I=\bigcup \{I_{y}|\ y\in
Y\}$.
\end{enumerate}
\end{definition}

\begin{theorem}\label{thGA2} \cite{Che_2022}, \cite[Ch.II]{Che_2024}
Let $Y$ be a compact metric space, $Y$ be invariant (i.e.,
$\sigma(t,Y)=Y$ for any $t\in \mathbb T$) and $\varphi$ be a
cocycle over $(Y,\mathbb T,\sigma)$ with the fiber $W$. If the
cocycle $\varphi$ is uniformly compactly dissipative, then it has
a compact global attractor.
\end{theorem}

\begin{definition}\label{defGA} Let $\langle W,\varphi,(Y,\mathbb T,\sigma)\rangle
$ be compactly dissipative, $K$ be the nonempty compact subset of
$W$ appearing in the equality (\ref{eqGA2.7.8}) and
$$
I_{y}=\omega_{y}(K):=\bigcap_{t\ge 0} \overline{\bigcup_{\tau\ge
t} \varphi(\tau,K,\sigma(-\tau,y))}
$$
for any $y\in Y$. The family of compact subsets $\{I_y|\ y\in Y\}$
is said to be a Levinson center (compact global attractor) of
non-autonomous (cocycle) dynamical system $\langle
W,\varphi,(Y,\mathbb T,\sigma)\rangle $.
\end{definition}

\begin{remark}\label{remGA} According to \cite{Che_2022} (see also \cite[Ch.II]{Che_2024}) by
Definition \ref{defGA} the notion of Levinson center (compact
global attractor) for non-autonomous (cocycle) dynamical system
$\langle W,\varphi,(Y,\mathbb T,\sigma)\rangle $ is well defined.
\end{remark}

\begin{definition}\label{defCGA1} A cocycle $\varphi$ is said to be
dissipative if there exists a bounded subset $K\subset \mathfrak
B$ such that for every bounded subset $B\subset \ \mathfrak B$
there exists a positive number $L=L(B)$ such that
$\varphi(t,B,Y)\subseteq K$ for all $t\ge L(B)$, where
$\varphi(t,B,Y):=\{\varphi(t,u,y)|\ (u,y)\in B\times Y\}$.
\end{definition}

\begin{theorem}\label{thCGA1} \cite[Ch.II]{Che_2024} Assume that the metric space $Y$ is
compact and the cocycle $\langle \mathfrak B,\varphi,(Y,\mathbb
R,\sigma)\rangle$ is dissipative and asymptotically compact.

Then the cocycle $\varphi$ has a compact global attractor.
\end{theorem}

\section{Finite Dimensional approximations of the system (\ref{eqI1}). }\label{Sec3}

In this section we consider the finite-dimensional approximations
of the infinite-dimensional system (\ref{eqI1}).

Let $\mathfrak f\in C(\mathbb R,l_{2})$ and $\mathfrak f
=(f_{i})_{i\in \mathbb Z}$. For every $n\in \mathbb N$ consider
the $2n+1$-dimensional system
\begin{equation}\label{eq1}
\left\{
\begin{array}{l}
\dot{v}_{-n} = -\nu(-v_{n} + 2v_{-n} - v_{-n+1}) - \lambda v_{-n} - F(v_{-n}) + f_{-n}(t) \\
\dot{v}_{-n+1} = -\nu(-v_{-n} + 2v_{-n+1} - v_{-n+2}) - \lambda v_{-n+1} - F(v_{-n+1}) + f_{-n+1}(t) \\
\vdots \\
\dot{v}_{i}=-\nu(-v_{i-1} + 2v_{i} - v_{i+1}) - \lambda
v_{i} - F(v_{i}) + f_{i}(t)\ \ (|i|< n-1)\\
\vdots \\
\dot{v}_{n-1} = -\nu(-v_{n-2} + 2v_{n-1} - v_{n}) - \lambda v_{n-1} - F(v_{n-1}) + f_{n-1}(t) \\
\dot{v}_{n} = -\nu(-v_{n-1} + 2v_{n} - v_{-n}) - \lambda v_{n} -
F(v_{n}) + f_{n}(t)
\end{array}
\right.
\end{equation}
$(v_{-n}, \dots, v_n)(0) = (v_{0,-n}, \dots, v_{0,n}) \in
\mathbb{R}^{2n+1}$.

Denote by $[\mathbb R^{n}]$ the space of all linear mapping
$A:\mathbb R^{n}\to \mathbb R^{n}$ equipped with the operator
norm. We can rewrite this system as follows. Let

$$
B_{n} =
\begin{pmatrix}
-1 & 1 & 0 & 0 & \cdots & 0 & 0\\
0 & -1 & 1 & 0 & \cdots & 0 & 0\\
\vdots & & \ddots & & \vdots \\
0 & 0 & 0 & 0 & \cdots & -1 & 1\\
1 & 0 & 0 & 0 & \cdots & 0 & -1
\end{pmatrix} \in [\mathbb{R}^{2n+1}]
$$

and denote by $A_{n}\in [\mathbb R^{2n+1}]$ defined by

$$
A_{n} =
\begin{pmatrix}
2 & -1 & 0 & 0 & 0 & \cdots & 0 & 0 & -1\\
-1 & 2 & -1 & 0 & 0 & \cdots & 0 & 0 & 0\\
0 & -1 & 2 & -1 & 0 & \cdots & 0 & 0 & 0\\
0 & 0 & 0 & 0 & 0 & \cdots & -1 & 2 & -1\\
-1 & 0 & 0 & 0 & 0 & \cdots & 0 & -1 & 2\\
\end{pmatrix} \in [\mathbb{R}^{2n+1}],
$$
then we have
$$
A_{n} = B_{n}^*B_{n} = B_{n}B_{n}^* \quad (B_{n}^* \text{ the
transpose of } B_{n})\ \mbox{for every}\ n\in \mathbb Z_{0}.
$$

Let $\mathfrak{f}\in C(\mathbb R,\ell_{2})$ and $n\in \mathbb N$
denote by $\mathfrak{f}^{n}=(f^{n}_{i})_{|i|\le n}\in C(\mathbb
R,\mathbb R^{2n+1})$ defined by
\begin{equation}\label{eqF1}
\mathfrak{f}^{n}=\left\{\begin{array}{ll}
&\!\! \mathfrak{f}^{n}_{i}=f_{i},\;\mbox{if}\; |i|\le n-1 \\[2mm]
&\!\! \mathfrak{f}^{n}_{n}=f_{-n-1}\ \mbox{and}\ \ f_{-n}=f_{n+1}.
\end{array}
\right.\nonumber
\end{equation}

The system of equation (\ref{eq1}) can be rewritten as follows
\begin{equation}\label{eq2}
\dot{v} = -\nu A_{n} v - \lambda v - f^{n}(v) +
\mathfrak{f}^{n}(t), \quad t \in \mathbb{T} \quad (\mathbb{T} =
\mathbb{R} \text{ or } \mathbb{R}_+)
\end{equation}
\begin{equation}\label{eq3}
v(0) = v_0 \in \mathbb{R}^{2n+1}
\end{equation}
where $v := (v_i)_{|i| \le n}$, $f^{n}(v) := (f(v_i))_{|i| \le n}$
and $\mathfrak{f}^{n}(t) := (\mathfrak{f}_i(t))_{|i| \le n}$

\begin{remark}\label{remI1}
1. Note that the space $\mathbb R^{2n+1}$ is embedded in the space
$\ell_{2}$ as follow:
$$
u\to \mathcal{I}_{n}(u):=(\ldots,0,u_{-n},u_{-n+1},\ldots,
u_{-1},u_{0},u_{1}.\dots,u_{n-1},u_{n},0,\ldots).
$$
It is evident that this embedding (the mapping $\mathcal{I}_{n}$)
is isometric and completely continuous.

2. Taking into account the fact above (see item 1.) we will
identify, when it is necessary, the space $\mathbb R^{2n+1}$ and
subspace $\mathcal{I}_{n}(\mathbb R^{2n+1})$ of the space
$\ell_{2}$.
\end{remark}

\section{Shift Dynamical Systems and Finite Dimensional
Approximations}\label{Sec4}

\begin{lemma}\label{lF1} The following statements hold:
\begin{enumerate}
\item[1.] $\mathfrak{f}^{n}\in C(\mathbb R,\mathbb R^{2n+1})$ for
every $\mathfrak{f}\in C(\mathbb R,\ell_{2})$ and $n\in \mathbb
N$; \item[2.] the mapping $P_{n}:C(\mathbb R,\ell_{2})\to
C(\mathbb R,\mathbb R^{2n+1})$ defined by
$P_{n}(\mathfrak{f}):=\mathfrak{f}^{n}$ possesses the following
properties:
\begin{enumerate}
\item[2.1] $P_{n}$ is continuous; \item[2.2]
\begin{equation}\label{eqH1}
P_{n}(\sigma(h,\mathfrak{f}))=\sigma(h,P_{n}(\mathfrak{f}))
\end{equation}
for all $(h,\mathfrak{f})\in \mathbb R\times C(\mathbb
R,\ell_{2})$, i.e., the mapping $P_{n}$ is a homomorphism of the
shift dynamical system $(C(\mathbb R,\ell_{2}),\mathbb R,\sigma)$
into the shift dynamical system $(C(\mathbb R,\mathbb
R^{2n+1}),\mathbb R,\sigma)$.
\end{enumerate}
\item[3.] $P_{n}(h(\mathfrak{f}))\subseteq H(P_{n}(\mathfrak{f}))$
for every $(n,\mathfrak{f})\in \mathbb N\times C(\mathbb
R,\ell_{2})$; \item[4.] if the function $\mathfrak{f}$ is Lagrange
stable, then for every $n\in \mathbb N$ we have
\begin{enumerate}
\item[4.1]
\begin{equation}\label{eqH_01}
P_{n}(H(\mathfrak{f}))=H(P_{n}(\mathfrak{f}))\ \ \mbox{and}
\end{equation}
\item[4.2] the function $\mathfrak{f}^{n}$ is also Lagrange
stable.
\end{enumerate}
\end{enumerate}
\end{lemma}
\begin{proof} The first statement of Lemma is evident.

Let $\mathfrak{f} =(f_{i})_{i\in\mathbb Z}\in C(\mathbb
R,\ell_{2})$, $\{\mathfrak{f}_{k}\}=\{(f_{k,i})_{i\in \mathbb
Z}\}\subset C(\mathbb R,\ell_{2})$ and $\mathfrak{f}_{k}\to
\mathfrak{f}$, i.e., for every $L>0$ we have
\begin{equation}\label{eqFC1.1}
\max\limits_{|t|\le L}\sum_{i\in \mathbb
Z}|f_{k,i}(t)-f_{i}(t)|^{2}\to 0
\end{equation}
as $k\to \infty$. Note that
\begin{eqnarray}\label{eqFC2}
& \max\limits_{|t|\le
L}|P_{n}(\mathfrak{f}_{k})(t)-P_{n}(\mathfrak{f})(t)|^{2}=
\max\limits_{|t|\le L}\sum_{|i|\le
n}|P_{n}(\mathfrak{f})_{i}(t)-P_{n}(\mathfrak{f})_{i}(t)|^{2} \le \nonumber \\
& \max\limits_{|t|\le L}\sum_{i\in \mathbb
Z}|f_{k,i}(t)-f_{i}(t)|^{2}
\end{eqnarray}
for all $k\in \mathbb N$. Passing to the limit in (\ref{eqFC2}) as
$k\to \infty$ and taking into account (\ref{eqFC1.1}) we obtain
\begin{equation}\label{eqFC_3}
\lim\limits_{k\to \infty}\max\limits_{|t|\le
L}|P_{n}(\mathfrak{f}_{k})(t)-P_{n}(\mathfrak{f})(t)|^{2}=0\nonumber
\end{equation}
for every $L>0$. This means that $P_{n}(\mathfrak{f}_k)\to
P_{n}(\mathfrak{f})$ in the space $C(\mathbb R,\mathbb R^{2n+1})$
as $k\to \infty$ and, consequently, the mapping $P_{n}$ is
continuous.

To establish the relation (\ref{eqH1}) we note that
\begin{equation}\label{eqFC4}
P_{n}(\sigma(h,\mathfrak{f}))=(f^{(h)}_{i})_{|i|\le n}\
\mbox{and}\ \sigma(h,P_{n}(\mathfrak{f}))=\sigma(h,(f_{i})_{|i|\le
n})=(f^{(h)}_{i})_{|i|\le n} \nonumber
\end{equation}
and, consequently,
$P_{n}(\sigma(h,\mathfrak{f}))=\sigma(h,P_{n}(\mathfrak{f}))$ for
all $(h,\mathfrak{f})\in \mathbb R\times C(\mathbb R,\mathbb
R^{2n+1})$ .

If $\mathfrak{g}\in H(\mathfrak{f})$, then there exists a sequence
$\{h_k\}\subset \mathbb R$ such that $\sigma(h_k,\mathfrak{f})\to
\mathfrak{g}$ as $k\to \infty$. Note that
\begin{equation}\label{eqH2}
P_{n}(\sigma(h_k,\mathfrak{f}))\to P_{n}(\mathfrak{g})
\end{equation}
as $k\to \infty$. On the other hand by the equality (\ref{eqH1})
we have
\begin{equation}\label{eqH3}
\sigma(h_k,P_{n}(\mathfrak{f}))=P_{n}(\sigma(h_k,\mathfrak{f}))
\end{equation}
for all $k\in \mathbb N$. Passing to the limit in (\ref{eqH3}) as
$k\to \infty$ and taking into account (\ref{eqH2}) we obtain
\begin{equation}\label{eqH4}
\lim\limits_{k\to
\infty}\sigma(h_k,P_{n}(\mathfrak{f}))=P_{n}(\mathfrak{g}),
\end{equation}
i.e., $P_{n}(\mathfrak{g})\in H(P_{n}(\mathfrak{f}))$. This means
that $P_{n}(H(\mathfrak{f}))\subseteq H(P_{n}(\mathfrak{f}))$.

Let now $\mathfrak{f}\in C(\mathbb R,\ell_{2})$ be Lagrange stable
and $\widetilde{\mathfrak{g}}\in H(P_{n}(\mathfrak{f}))$ then
there exists a sequence $\{h_k\}\subset \mathbb R$ such that
\begin{equation}\label{eqH5}
\sigma(h_{k},P_{n}(\mathfrak{f}))=P_{n}(\mathfrak{f})^{(h_k)}\to
\widetilde{\mathfrak{g}} \nonumber
\end{equation}
as $k\to \infty$. Since the function $\mathfrak{f}\in C(\mathbb
R,\ell_{2})$ is Lagrange stable then without loss of generality we
can assume that the sequence $\{\mathfrak{f}^{(h_k)}\}\subset
C(\mathbb R,\ell_{2})$ converges. Denote its limit by
$\mathfrak{g}:=\lim\limits_{k\to \infty}\mathfrak{f}^{(h_k)}$ then
taking into account (\ref{eqH3})-(\ref{eqH4}) we obtain
\begin{equation}\label{eqH6}
P_{n}(\mathfrak{g})=\lim\limits_{k\to
\infty}P_{n}(\mathfrak{f}^{(h_k)})=\lim\limits_{k\to
\infty}\sigma(h_k,P_{n}(\mathfrak{f}))=\widetilde{\mathfrak{g}},
\nonumber
\end{equation}
i.e., $P_{n}(H(\mathfrak{f}))=H(P_{n}(\mathfrak{f}))$.

Finally, we note that if the function $\mathfrak{f}\in C(\mathbb
R,\ell_{2})$ is Lagrange stable, then for every $n\in \mathbb N$
by (\ref{eqH_01}) we have
$P_{n}(H(\mathfrak{f}))=H(P_{n}(\mathfrak{f}))$. By the second
statement of Lemma the mapping $P_{n}: C(\mathbb R,\ell_{2})\to
C(\mathbb R,\mathbb R^{(2n+1)})$ is continuous and, consequently,
the set $H(P_{n}(\mathfrak{f}))=P_{n}(H(\mathfrak{f}))$ is a
compact subset of $C(\mathbb R,\mathbb R^{2n+1})$ because the set
$H(\mathfrak{f})$ is compact. Lemma is completely proved.
\end{proof}

\begin{lemma}\label{lF2} The following statements hold:
\begin{enumerate}
\item $\mathfrak{f}^{n}\in C(\mathbb R,\mathbb R^{2n+1})$ for
every $\mathfrak{f}\in C(\mathbb R,\ell_{2})$ and $n\in \mathbb
N$; \item the mapping $H_{n}:C(\mathbb R,\ell_{2})\to C(\mathbb
R,\mathbb R^{2n+1})$ defined by
$H_{n}(\mathfrak{f}):=\mathfrak{f}^{n}$ possesses the following
properties:
\begin{enumerate}
\item $H_{n}$ is continuous; \item
\begin{equation}\label{eqH1_1}
H_{n}(\sigma(h,\mathfrak{f}))=\sigma(h,H_{n}(\mathfrak{f}))
\end{equation}
for all $(h,\mathfrak{f})\in \mathbb R\times C(\mathbb
R,\ell_{2})$, i.e., the mapping $H_{n}$ is a homomorphism of the
shift dynamical system $(C(\mathbb R,\ell_{2}),\mathbb R,\sigma)$
into the shift dynamical system $(C(\mathbb R,\mathbb
R^{2n+1}),\mathbb R,\sigma)$.
\end{enumerate}
\item $P_{n}(h(\mathfrak{f}))\subseteq H(P_{n}(\mathfrak{f}))$ for
every $(n,\mathfrak{f})\in \mathbb N\times C(\mathbb R,\ell_{2})$;
\item if the function $\mathfrak{f}$ is Lagrange stable, then for
every $n\in \mathbb N$ we have
\begin{enumerate}
\item
\begin{equation}\label{eqH_1}
H_{n}(H(\mathfrak{f}))=H(H_{n}(\mathfrak{f}))\ \
\mbox{and}\nonumber
\end{equation}
\item the function $\mathfrak{f}^{n}$ is also Lagrange stable.
\end{enumerate}
\end{enumerate}
\end{lemma}
\begin{proof} The first statement of Lemma is trivial.

Let $\mathfrak{f} =(f_{i})_{i\in\mathbb Z}\in C(\mathbb
R,\ell_{2})$, $\{\mathfrak{f}_{k}\}=\{(f_{k,i})_{i\in \mathbb
Z}\subset C(\mathbb R,\ell_{2})$ and $\mathfrak{f}_{k}\to
\mathfrak{f}$, i.e., for every $L>0$ we have
\begin{equation}\label{eqFC1}
\max\limits_{|t|\le L}\sum_{i\in \mathbb
Z}|f_{k,i}(t)-f_{i}(t)|^{2}\to 0 \nonumber
\end{equation}
as $k\to \infty$.

Since
\begin{eqnarray}\label{eqFC2_1}
&  \max\limits_{|t|\le
L}|H_{n}(\mathfrak{f}_{k})(t)-H_{n}(\mathfrak{f})(t)|^{2}=
\max\limits_{|t|\le L}\sum_{|i|\le
n}|H_{n}(\mathfrak{f})_{i}(t)-H_{n}(\mathfrak{f})_{i}(t)|^{2}=\nonumber \\
& \max\limits_{|t|\le L}\sum_{|i|\le
n-1}|H_{n}(\mathfrak{f})_{i}(t)-H_{n}(\mathfrak{f})_{i}(t)|^{2} + \nonumber \\
& \max\limits_{|t|\le L}\sum_{|i|=
n}|H_{n}(\mathfrak{f})_{i}(t)-H_{n}(\mathfrak{f})_{i}(t)|^{2}
\end{eqnarray}
and
\begin{equation}\label{eqFC2_2}
\max\limits_{|t|\le
L}\sum\limits_{|i|=n}|H_{n}(\mathfrak{f})_{i}(t)-H_{n}(\mathfrak{f})_{i}(t)|^{2}=
\max\limits_{|t|\le
L}\sum\limits_{|i|=n}|f_{k,i}(t)-f_{i}(t)|^{2}=
\end{equation}
$$
\max\limits_{|t|\le L}|f_{k,n+1}(t)-f_{n+1}(t)|^{2}
+\max\limits_{|t|\le L}|f_{k,-n-1}(t)-f_{-n-1}(t)|^{2} \le
$$
$$
2\max\limits_{|t|\le L}\sum\limits_{|i|\le
n+1}|f_{k,i}(t)-f_{i}(t)|^{2}\le 2\max\limits_{|t|\le
L}\sum\limits_{i\in \mathbb
Z}|f_{k,i}(t)-f_{i}(t)|^{2}=
$$
$$
2\max\limits_{|t|\le L}\|\mathfrak{f}_{k}(t)-\mathfrak{f}(t)\|\to
0
$$
as $k\to \infty$.

From (\ref{eqFC2_1}) and (\ref{eqFC2_2}) we receive
\begin{equation}\label{eqFC2_3}
\max\limits_{|t|\le
L}|H_{n}(\mathfrak{f}_{k})(t)-H_{n}(\mathfrak{f})(t)|^{2}\to
0\nonumber
\end{equation}
as $k\to \infty$ for every $L>0$, i.e., $H_{n}(\mathfrak{f}_k)\to
H_{n}(\mathfrak{f})$ in the space $C(\mathbb R,\mathbb R^{2n+1})$
and, consequently, the mapping $H_{n}$ is continuous.

To prove (\ref{eqH1_1}) we note that
\begin{equation}\label{eqFC2_4}
H_{n}(\sigma(h,\mathfrak{f}))=H_{n}(\mathfrak{f}^{(h)})(t)=
\left\{\begin{array}{ll}
&\!\! f_{i}^{(h)}(t),\;\mbox{if}\; |i|\le n-1 \\[2mm]
&\!\! f_{-n-1}^{(h)}(t),\; \mbox{if}\; i=n\ \ \mbox{and}\ \\[2mm]
&\!\! f_{n+1}^{(h)}(t),\; \mbox{if}\; i=-n
\end{array}
\right.\nonumber
\end{equation}
and
\begin{equation}\label{eqFC2_5}
\sigma(h,H_{n}(\mathfrak{f}))=H_{n}(\mathfrak{f})^{(h)}(t)=
\left\{\begin{array}{ll}
&\!\! f_{i}^{(h)}(t),\;\mbox{if}\; |i|\le n-1 \\[2mm]
&\!\! f_{-n-1}^{(h)}(t),\; \mbox{if}\; i=n\ \ \mbox{and}\ \\[2mm]
&\!\! f_{n+1}^{(h)}(t),\; \mbox{if}\; i=-n
\end{array}
\right.\nonumber
\end{equation}
for all $t\in \mathbb R$ and, consequently,
$H_{n}(\sigma(h,\mathfrak{f}))=\sigma(h,H_{n}(\mathfrak{f}))$ for
every $(h,\mathfrak{f})\in \mathbb R\times C(\mathbb R,\mathbb
R^{2n+1})$ .

Let $\mathfrak{g}\in H(\mathfrak{f})$ then there exists a sequence
$\{h_k\}\subset \mathbb R$ such that $\sigma(h_k,\mathfrak{f})\to
\mathfrak{g}$ as $k\to \infty$. Note that
\begin{equation}\label{eqH2_1}
H_{n}(\sigma(h_k,\mathfrak{f}))\to H_{n}(\mathfrak{g})
\end{equation}
as $k\to \infty$. By the equality (\ref{eqH1_1}) we have
\begin{equation}\label{eqH3_1}
\sigma(h_k,H_{n}(\mathfrak{f}))=H_{n}(\sigma(h_k,\mathfrak{f}))
\end{equation}
for all $k\in \mathbb N$. Passing to the limit in (\ref{eqH3_1})
as $k\to \infty$ and taking into account (\ref{eqH2_1}) we obtain
\begin{equation}\label{eqH4_1}
\lim\limits_{k\to
\infty}\sigma(h_k,H_{n}(\mathfrak{f}))=H_{n}(\mathfrak{g}),
\end{equation}
i.e., $H_{n}(\mathfrak{g})\in H(H_{n}(\mathfrak{f}))$. This means
that $H_{n}(H(\mathfrak{f}))\subseteq H(H_{n}(\mathfrak{f}))$.

If the function $\mathfrak{f}\in C(\mathbb R,\ell_{2})$ is
Lagrange stable, then for every $\widetilde{\mathfrak{g}}\in
H(P_{n}(\mathfrak{f}))$ there exists a sequence $\{h_k\}\subset
\mathbb R$ such that
\begin{equation}\label{eqH5_1}
H_{n}(\mathfrak{f})^{(h_k)}\to \widetilde{\mathfrak{g}} \nonumber
\end{equation}
as $k\to \infty$. Without loss of generality we can assume that
the sequence $\{\mathfrak{f}^{(h_k)}\}\subset C(\mathbb
R,\ell_{2})$ converges because the function $\mathfrak{f}\in
C(\mathbb R,\ell_{2})$ is Lagrange stable. Denote its limit by
$\mathfrak{g}:=\lim\limits_{k\to \infty}\mathfrak{f}^{(h_k)}$ then
taking into account (\ref{eqH3_1})-(\ref{eqH4_1}) we obtain
\begin{equation}\label{eqH6_1}
H_{n}(\mathfrak{g})=\lim\limits_{k\to
\infty}H_{n}(\mathfrak{f}^{(h_k)})=\lim\limits_{k\to
\infty}\sigma(h_k,H_{n}(\mathfrak{f}))=\widetilde{\mathfrak{g}},
\nonumber
\end{equation}
i.e., $P_{n}(H(\mathfrak{f}))=H(P_{n}(\mathfrak{f}))$.

We note that if the function $\mathfrak{f}\in C(\mathbb
R,\ell_{2})$ is Lagrange stable, then for every $n\in \mathbb N$
by (\ref{eqH_01}) we have
$H_{n}(H(\mathfrak{f}))=H(H_{n}(\mathfrak{f}))$. By the second
statement of Lemma the mapping $H_{n}: C(\mathbb R,\ell_{2})\to
C(\mathbb R,\mathbb R^{2n+1})$ is continuous and, consequently,
the set $H(H_{n}(\mathfrak{f}))=H_{n}(H(\mathfrak{f}))$ is a
compact subset of $C(\mathbb R,\mathbb R^{2n+1})$ because the set
$H(\mathfrak{f})$ is compact. Lemma is proved.
\end{proof}

\begin{lemma}\label{lF_3} Let $\mathfrak{f}\in C(\mathbb R,\ell_{2})$ and $L$
be a given positive number then for every $\varepsilon >0$ there
exists a natural number $k_{\varepsilon}\in \mathbb N$ such that
\begin{equation}\label{eqC1_1}
\sum\limits_{|i|>k_{\varepsilon}}|f_{i}(t)|^{2}<\varepsilon
\nonumber
\end{equation}
for every $t\in [-L,L]$, i.e.,
\begin{equation}\label{eqC1_2}
\sup\limits_{|t|\le L}R_{n}(\mathfrak{f})(t)=0,
\end{equation}
where
\begin{equation}\label{eqC1_3}
R_{n}(\mathfrak{f})(t)=\sum\limits_{|i|\ge
n+1}|f_{i}(t)|^{2}\nonumber
\end{equation}
for all $t\in \mathbb R$ and $n\in \mathbb N$.
\end{lemma}
\begin{proof} Assume that $\mathfrak{f}\in C(\mathbb R,\ell_{2})$ and $L>0$
then $Q_{L}:=\mathfrak{f}([-L,L])$ is a compact subset of
$\ell_{2}$. By Theorem 5.25 \cite[Ch.V, p.167]{LS_1974} for every
$\varepsilon
>0$ there exists a number $k_{\varepsilon}\in \mathbb N$ such that
\begin{equation}\label{eqC1_4}
\sum\limits_{|i|>k_{\varepsilon}}|w_{i}|^{2}<\varepsilon \nonumber
\end{equation}
for all $w\in Q_{L}$. Let now $t\in [-L,L]$ then for
$w:=\mathfrak{f}(t)\in Q_{L}=\mathfrak{f}([-L,L])$ we have
\begin{equation}\label{eqC1_5}
\sum\limits_{|i|>k_{\varepsilon}}|f_{i}(t)|^{2}=\sum\limits_{|i|>k_{\varepsilon}}|w_{i}|^{2}<\varepsilon
.\nonumber
\end{equation}
\end{proof}

\begin{remark}\label{remF1} Note that the relation (\ref{eqC1_2})
remains true if we replace $\sup$ by $\max$ because the function
$R_{n}(\mathfrak{f})\in C(\mathbb R,\mathbb R_{+})$, where
\begin{equation}\label{eqC1_6}
R_{n}(\mathfrak{f})(t)=\sum\limits_{|i|\ge
n+1}|f_{i}(t)|^{2}=\|\mathfrak{f}(t)\|^{2}-\sum\limits_{|i|\le
n}|f_{i}(t)|^{2} \nonumber
\end{equation}
for all $(n,t)\in \mathbb N\times \mathbb R$.
\end{remark}

\begin{lemma}\label{lF3} The following statements hold:
\begin{enumerate}
\item for every $\mathfrak{f}\in C(\mathbb R,\ell_{2})$ we have
$P_{n}(\mathfrak{f})\to \mathfrak{f}$ in $C(\mathbb R,\ell_{2})$
as $n\to \infty$; \item if the function $\mathfrak{f}\in C(\mathbb
R,\ell_{2})$ is Lagrange stable, then
$$
H(K):=\overline{\{\sigma(h,\mathfrak{g})|\ h\in\mathbb R,\
\mathfrak{g}\in K\}}
$$
is a compact subset of $C(\mathbb R,\ell_{2})$, where
$K:=\{P_{n}(\mathfrak{f})|\ n\in \mathbb N\}$.
\end{enumerate}
\end{lemma}
\begin{proof} Let $\mathfrak{f}\in C(\mathbb R,\ell_{2})$ then by Lemma \ref{lF_3} we have
\begin{equation}\label{eqC10}
\lim_{n\to \infty}\max\limits_{|t|\le L}\sum\limits_{|i|\ge
n+1}|f_{i}(t)|^{2}=0 .
\end{equation}
Note that
\begin{equation}\label{eqC11}
\max\limits_{|t|\le
L}\|P_{n}(\mathfrak{f})(t)-\mathfrak{f}(t)\|^{2}=\max\limits_{|t|\le
L}\sum\limits_{i\in \mathbb
Z}|P_{n,i}(f)(t)-f_{i}(t)|^{2}=\max\limits_{|t|\le
L}\sum\limits_{|i|\ge n+1}|f_{i}(t)|^{2}
\end{equation}
for every $n\in \mathbb N$. Passing to the limit in (\ref{eqC11})
and taking into account (\ref{eqC10}) we obtain that
$P_{n}\mathfrak{f}\to \mathfrak{f}$ in the space $C(\mathbb
R,\ell_{2})$.

Assume that the function $\mathfrak{f}\in C(\mathbb R,\ell_{2})$
is Lagrange stable. By Lemma \ref{lAPF02} the following statements
hold:
\begin{enumerate}
\item there exists a compact subset $Q\subset \ell_{2}$ such that
\begin{equation}\label{eqC12}
\mathfrak{f}(\mathbb R)\subset Q;\nonumber
\end{equation}
\item for every $\varepsilon >0$ there exists a $\delta
=\delta(\varepsilon)>0$ such that
\begin{equation}\label{eqC13}
|t_1-t_2|<\delta\ \ \mbox{implies}\ \
\|f(t_1)-f(t_2)\|^{2}<\varepsilon^{2}/2 .
\end{equation}
\end{enumerate}

Since the subset $\mathfrak{f}(\mathbb R)\subset \ell_{2}$ is
precompact then by Theorem 5.25 \cite[Ch.V, p.167]{LS_1974} for
$\varepsilon
>0$ there exists a number $k_{\varepsilon}\in \mathbb N$ such that
\begin{equation}\label{eqC14}
\sum\limits_{|i|\ge
k_{\varepsilon}}|f_{i}(h)|^{2}<\varepsilon^{2}/8
\end{equation}
for all $h\in \mathbb R$.

Consider the family of functions $\mathbb F
:=\{\sigma(h,P_{n}(\mathfrak{f}))|\ (h,n)\in \mathbb R\times
\mathbb N\}$. Let $\psi \in \mathbb F$ then there exists
$(h_0,n_0)\in \mathbb R\times \mathbb N$ such that $\psi
=\sigma(h_0,P_{n_0}(\mathfrak{f}))$. Note that
\begin{eqnarray}\label{eqC15}
& \sum\limits_{|i|\ge
k_{\varepsilon}}|\psi(\tau)|^{2}=\sum\limits_{|i|\ge
k_{\varepsilon}}|\sigma(h_0,P_{n_{0}}(\mathfrak{f}))(\tau)|^{2}=\\
& \sum\limits_{k_{\varepsilon}\le |i|\le
n_{0}}|f^{(h)}_{i}(\tau)|^{2}= \sum\limits_{k_{\varepsilon}\le
|i|\le n_{0}}|f_{i}(h_0+\tau)|^{2}\le \sum\limits_{|i|\ge
k_{\varepsilon}}|f_{i}(h_0+\tau)|^{2}\nonumber
\end{eqnarray}
for all $\tau\in \mathbb R$. From (\ref{eqC14}) and (\ref{eqC15})
we receive
\begin{equation}\label{eqC16}
\sum\limits_{|i|\ge k_{\varepsilon}}|\psi(\tau)|^{2} <\varepsilon
\nonumber
\end{equation}
for all $\tau \in \mathbb R$ and, consequently, $\mathbb F(\mathbb
R):=\{\psi(\tau)|\ \psi \in \mathbb F,\ \tau \in \mathbb R\}$ is a
precompct subset of $\ell_{2}$.

Let $\varepsilon$ be an arbitrary positive number and $\delta
=\delta(\varepsilon)>0$ the number figuring in (\ref{eqC13}). If
$\psi\in \mathbb F$, then there exist $h\in \mathbb R$ and $n\in
\mathbb N$ such that $\psi
=\sigma(h,P_{n}(\mathfrak{f}))=P_{n}(\mathfrak{f})^{(h)}$. Let
$t_1,t_2\in \mathbb R$ be two arbitrary numbers with
$|t_1-t_2|<\delta$.

We will consider two cases:

a. $n< k_{\varepsilon}$, then it is clear that there exists
$0<\delta_{1}(\varepsilon)<\delta(\varepsilon)$ such that
$|t_1-t_2|<\delta_{1}(\varepsilon)$ implies
\begin{equation}\label{eqC_17}
\|P_{m}(\mathfrak{f})^{(h)}(t_1)-P_{m}(\mathfrak{f})^{(h)}(t_2)\|<\varepsilon
\nonumber
\end{equation}
for every $1\le m\le n$.

b. If $n\ge k_{\varepsilon}$, then taking into account
(\ref{eqC13})-(\ref{eqC14}) we obtain

\begin{eqnarray}\label{eqC17}
& \|\psi(t_1)-\psi(t_2)\|^{2}=\sum\limits_{i\in \mathbb
Z}|\psi_{i}(t_1)-\psi_{i}(t_2)|^{2}=\nonumber \\
& \sum\limits_{|i|\le
k_{\varepsilon}-1}|P_{n}(\mathfrak{f})^{(h)}_{i}(t_1)-P_{n}(\mathfrak{f})^{(h)}_{i}(t_2)|^{2}+
\sum\limits_{|i|\ge
k_{\varepsilon}}|P_{n}(\mathfrak{f})^{(h)}_{i}(t_1)-P_{n}(\mathfrak{f})^{(h)}_{i}(t_2)|^{2}
= \nonumber
\\
& \sum\limits_{|i|\le
k_{\varepsilon}-1}|f^{(h)}_{i}(t_1)-f^{(h)}_{i}(t_2)|^{2}+
\sum\limits_{k_{\varepsilon}\le |i|\le
n}|f^{(h)}_{i}(t_1)-f^{(h)}_{i}(t_2)|^{2} \le \nonumber
\\
& \sum\limits_{|i|\le
k_{\varepsilon}-1}|f^{(h)}_{i}(t_1)-f^{(h)}_{i}(t_2)|^{2}+
2(\sum\limits_{k_{\varepsilon}\le |i|\le
n}|f^{(h)}_{i}(t_1)|^{2}+|f^{(h)}_{i}(t_2)|^{2})\le \nonumber \\
& \sum\limits_{|i|\ge
k_{\varepsilon}-1}|f_{i}(t_1)-f_{i}(t_2)|^{2}+2(\sum\limits_{|i|\ge
k_{\varepsilon}}|f^{(h)}_{i}(t_1)|^{2}+\sum\limits_{|i|\ge
k_{\varepsilon}} |f^{(h)}_{i}(t_2)|^{2}))<\nonumber
\\
&
\varepsilon^{2}/2+2(\varepsilon^{2}/8+\varepsilon^{2}/8)=\varepsilon^{2}\nonumber
.
\end{eqnarray}
Now to finish the proof of Lemma it suffices to apply Lemma
\ref{lAPF02}.
\end{proof}

\begin{lemma}\label{lF4} Let $\mathfrak{f}\in C(\mathbb R,\ell_{2})$ be a Lagrange stable
function, $K:=\{P_{n}(\mathfrak{f})\}_{n\in \mathbb Z_{0}}$ and
$H(K):=\overline{\{\sigma(h,P_{n}(\mathfrak{f}))|\ h\in \mathbb
R,\ n\in \mathbb Z_{0}\}}$. Then there exists a positive constant
$C$ such that
\begin{equation}\label{eqGF1}
\|\mathfrak{g}(t)\|\le C \nonumber
\end{equation}
for all $t\in \mathbb R$ and $\mathfrak{g}\in H(K)$.
\end{lemma}
\begin{proof} By Lemma \ref{lF3} the set $H(K)$ is a compact in
the space $C(\mathbb R,\ell_{2})$. By Lemma \ref{lAPF02} there
exists a compact $Q$ from $\ell_{2}$ such that
\begin{equation}\label{eqGF2}
\mathfrak{g}(\mathbb R)\subseteq Q \nonumber
\end{equation}
for all $t\in \mathbb R$ and $\mathfrak{g}\in H(K)$. Then we have
\begin{equation}\label{eqGF3}
C:=\sup\{\|\mathfrak{g}(t)\|:\ t\in \mathbb R,\ \mathfrak{g}\in
H(K)\}<+\infty . \nonumber
\end{equation}
Lemma is proved.
\end{proof}

\section{Cocycles Generated by Nonautonomous Lattice Dynamical
Systems}\label{Sec5}

Below we use the following conditions.

\emph{Condition }(\textbf{C1}). \label{C1} The function
$\mathfrak{f}\in C(\mathbb R,\mathfrak B)$ and it is
translation-compact, i.e., the set $\{\mathfrak{f}^{h}|\ h\in
\mathbb R\}$ is pre-compact in the space $C(\mathbb R,\mathfrak
B)$.

\emph{Condition} (\textbf{C2}). \label{C2} The function $F\in
C(\mathbb R,\mathbb R)$ is Lipschitz continuous on bounded sets
and $F(0)=0$.

Denote by $ \widetilde{F}:\ell_{2}\to \ell_{2}$ the Nemytskii
operator generated by $F$, i.e.,
$\widetilde{F}(\xi)_{i}:=F(\xi_{i})$ for any $i\in \mathfrak N$.

\emph{Condition} (\textbf{C3}). \label{C3} $sF(s)\leq -\alpha s^2$
for any $s\in \mathbb R$.

\begin{definition}\label{defL1.8} A function $F\in C(Y\times \mathfrak B,\mathfrak
B)$ is said to be globally Lipschitzian (respectively locally
Lipschitzian) w.r.t. variable $u\in \mathfrak B$
uniformly w.r.t. $y\in Y$ if there exists a positive
constant $L$ (for any bounded set $B \subset \mathfrak B$ there
exists a constant $L_{B}$) such that
\begin{equation}\label{eqL1.81}
\|F(y,u_1)-F(y,u_2)\|\le L \|u_1-u_2\|
\end{equation}
(respectively,
\begin{equation}\label{eqL1.82}
\|F(y,v_1)-F(y,v_2)\|\le L_{B} \|v_1-v_2\|)
\end{equation}
for any $u_1,u_2\in \mathfrak B$ and $y\in Y$ (respectively $v_1,
v_2 \in B \subset \mathfrak B$ and $y\in Y$).
\end{definition}

\begin{definition}\label{defL2.8}
The smallest constant $L$ (respectively $L_{B}$) with the property
(\ref{eqL1.81}) (respectively, (\ref{eqL1.82})) is called
Lipshchitz constant of function $F$ (notation $Lip(F)$,
respectively $Lip_{B}(F)$).
\end{definition}

Let $B \subset \mathfrak B$, denote by $CL(Y\times B,\mathfrak B)$
the Banach space of any Lipschitzian functions $F\in C(Y\times
B,\mathfrak B)$ equipped with the norm
\begin{equation}
||F||_{CL}:=\max\limits_{y\in Y}\|F(y,0)\|+Lip_{B}(F).\nonumber
\end{equation}

\begin{lemma}\label{l2.2} \cite{BLW_2001} Under the Condition (\textbf{C2})
it is well defined the mapping $\widetilde{F}:\ell_{2}\to
\ell_{2}$ and
\begin{equation}\label{eq2.2}
\|\widetilde{F}(\xi)-\widetilde{F}(\eta)\|\le Lip_{B}(F)\|\xi
-\eta \| \nonumber
\end{equation}
for any $\xi,\eta\in \ell_{2}$, where $\|\cdot\|^{2}:=\langle
\cdot,\cdot \rangle$ and $\|\cdot \|$ is the norm on the space
$\ell_{2}$.
\end{lemma}

Under the Conditions (\textbf{C1}) and (\textbf{C2}) the system of
differential equations (\ref{eqI1}) can be written in the form of
an ordinary differential equation
\begin{equation}\label{eq2.3}
u'=\nu A u +\Phi (u)+\mathfrak{f}(t)
\end{equation}
in the Banach space $\mathfrak B=\ell_{2}$, where
$\Phi(u):=-\lambda u +\widetilde{F} (u)$ and
$(Au)_{i}:=u_{i-1}-2u_{i}+u_{i+1}$ for every $u=(u_i)_{i\in
\mathbb Z}\in \ell_{2}$. Along with equation (\ref{eq2.3}) we
consider also it $H$-class, i.e., the family equations
\begin{equation}\label{eq2.3g}
u'=\nu A u +\Phi (u)+\widetilde{\mathfrak{f}}(t),
\end{equation}
where $\widetilde{\mathfrak{f}}\in H(\mathfrak{f})$ (respectively,
$\widetilde{\mathfrak{f}}\in H(K)$).

Let $\mathbb F$ be a compact subset of $C(\mathbb R,\ell_{2})$ and
invariant w.r.t. shift dynamical system $(C(\mathbb
R,\ell_{2}),\mathbb R,\sigma)$. Consider a family of differential
equations
\begin{equation}\label{eq2.4}
u'=\nu A u +\Phi (u)+\mathfrak{g}(t),
\end{equation}
where $\mathfrak{g}\in \mathbb F$.

The family of equations (\ref{eq2.4}) can be rewritten as follows
\begin{equation}\label{eq2.5}
u'=F(\sigma(t,\mathfrak{g}),u)\ \ (\mathfrak{g}\in \mathbb
F)),\nonumber
\end{equation}
where $F:\mathbb F\times \ell_{2}\to \ell_{2}$ is defined by
$F(\mathfrak{g},u):=\nu A u+\Phi (u) +\mathfrak{g}(0)$ for every
$(\mathfrak{g},u)\in \mathbb F\times \ell_{2}$. It easy to see
that $F(\sigma(t,\mathfrak{g}),u)=\nu A u+\Phi
(u)+\mathfrak{g}(t)$ for all $(t,u,\mathfrak{g})\in \mathbb
R\times \mathfrak B \times \mathbb F$.

Let $Y$ be a complete metric space, $(Y,\mathbb R,\sigma)$ be a
dynamical system on $Y$ and $ \mathfrak{A} $ be some complete
metric space of linear closed operators acting into Banach space $
\mathfrak B $. Consider the following linear differential equation
\begin{equation}\label{eqLS01.81}
x'=A(\sigma(t,y))x,\  \ (y\in Y)
\end{equation}
where $A\in C(Y,\mathfrak{A})$. We assume that the following
conditions are fulfilled for the equation (\ref{eqLS01.81}):
\begin{enumerate}
\item[a.] for every $ u \in \mathfrak B $ and $y\in Y $ the
equation (\ref{eqLS01.81}) has exactly one solution that is
defined on $ \mathbb R_{+} $ and satisfies the condition $ \varphi
(0,u,y) = u ;$ \item[b.] the mapping $ \varphi : (t,u,y) \to
\varphi (t,u,y) $ is continuous in the topology of $ \mathbb R_{+}
\times \mathfrak B \times Y$.
\end{enumerate}

Denote by $U(t,y):=\varphi(t,\cdot,y)$ for all $(t,y)\in \mathbb
R_{+}\times Y$.

Consider an evolutionary differential equation
\begin{equation}\label{eqSL1}
u'=A(\sigma(t,y))u + F(\sigma(t,y),u) \ \ (y\in Y)
\end{equation}
in the Banach space $\mathfrak B$, where $F$ is a nonlinear
continuous mapping ("small" perturbation) acting from $Y\times
\mathfrak B$ into $\mathfrak B$.

\begin{definition}
A function $u:[0,a)\mapsto \mathfrak B$ is said to be a weak
(mild) solution of the equation (\ref{eqSL1}) passing through the
point $x\in \mathfrak B$ at the initial moment $t=0$ (notation
$\varphi(t,u,y)$) if $u\in C([0,T],\mathfrak B)$ and satisfies the
integral equation
\begin{equation}\label{eqSL3}
u(t)=U(t,y)u+\int_{0}^{t}U(t-s,\sigma(s,y))F(\sigma(s,y),u(s))ds\nonumber
\end{equation}
for every $t\in [0,T]$ and $0<T<a$.
\end{definition}

\begin{theorem}\label{thLS1.S1} \cite[Ch.VI]{Che_2020}
Suppose that the function $F\in C(Y\times \mathfrak B,\mathfrak
B)$ is locally Lipschitzian. Let $x_0\in \mathfrak B$, $r>0 $ and
the conditions listed above be fulfilled. Then, there exist
positive numbers $\delta =\delta (x_0,r)$ and $T=T(x_0,r) $ such
that the equation (\ref{eqSL1}) admits a unique solution $\varphi
(t,u,y)$ ($u\in B[u_0,\delta]:=\{u\in \mathfrak B \ | \ \| u -u_0
\| \le \delta \} $) defined on the interval $[0,T]$ with the
conditions: $\varphi (0,u,y)=u$, $\| \varphi (t,u,y)-u_0\| \le r$
for all $t\in [0,T]$ and the mapping $ \varphi : [0,T]\times
B[u_0,\delta] \times Y \to \mathfrak B\ ( (t,u,y)\mapsto \varphi
(t,u,y))$ is continuous.
\end{theorem}

\begin{remark}\label{remSL1} Under the conditions of Theorem
\ref{thLS1.S1} the following statements hold \cite[Ch.IV]{Bur}:
\begin{enumerate}
\item if $\psi$ is a solution of the equation (\ref{eqSL1}) on
some interval $[0,h]$, then $\psi$ can be extended over a maximal
interval of existence $[0,\alpha)$; \item if the solution $\psi$
is bounded, then $\psi$ can be extended on the interval
$[0,+\infty)$.
\end{enumerate}
\end{remark}

\begin{theorem}\label{th1.1} Under the Conditions (\textbf{C1}) and
(\textbf{C2}) there exist positive numbers $\delta =\delta
(u_0,r)$ and $T=T(u_0,r) $ such that the equation (\ref{eqSL1})
admits a unique solution $\varphi (t,u,g)$ ($u\in B[u_0,\delta]$)
defined on the interval $[0,T]$ with the conditions: $\varphi
(0,u,g)=u$, $\|\varphi (t,u,\mathfrak{g})-u_0\| \le r$ for every
$t\in [0,T]$ and the mapping $ \varphi : [0,T]\times B[u_0,\delta]
\times \mathbb F \to \ell_{2}\ ( (t,u,\mathfrak{g})\mapsto \varphi
(t,u,\mathfrak{g}))$ is continuous.
\end{theorem}
\begin{proof}
Assume that the Conditions (\textbf{C1}) and (\textbf{C2}) are
fulfilled. Consider the equation (\ref{eq2.4}), where
$F(\mathfrak{g},u):=\nu A u +\Phi(u)+\mathfrak{g}(0)$ for all
$(u,\mathfrak{g})\in \ell_{2}\times \mathbb F$. It easy to check
that under the conditions of Theorem the mapping $F$ possesses the
following properties:
\begin{enumerate}
\item $F$ is continuous; \item the mapping $F$ is locally
Lipschitzian in $u\in \ell_{2}$ uniformly w.r.t. $\mathfrak{g}\in
\mathbb F$, i.e., for every bounded subset $B\subset \ell_{2}$
there exists a positive constant $L_{F}(B)$ such that
\begin{equation}\label{eqLB1}
\|F(u_1,\mathfrak{g})-F(u_2,\mathfrak{g})\|\le
L_{F}(B)\|u_1-u_2\|\nonumber
\end{equation}
for all $u_1,u_2\in B$ and $g\in \mathbb F$; \item there exists a
positive constant $A$ such that
\begin{equation}\label{eqA1}
\|F(\mathfrak{g},0)\|\le C\nonumber
\end{equation}
for every $\mathfrak{g}\in \mathbb F$.
\end{enumerate}

Now to finish the proof of Theorem it suffices to apply Theorem
\ref{thLS1.S1}.
\end{proof}

\begin{lemma}\label{lBC_1} Assume that the conditions (\textbf{C1})--(\textbf{C3}) holds and $\mathfrak{g} \in \mathbb F$.
Then, for every $T > 0$, the solution $v(t)$ of the problem
(\ref{eq2.4}) and $v(0)=v_0 \in \ell_{2}$ satisfies
\[
\|v(t)\| \le M \;,\quad \text{for all } 0 \le t \le T\,,
\]
where $M$ is a constant depending only on the data $(\lambda, C,
\|v_0\|)$ and $T$, where $C:=\sup\{\|\mathfrak{g}(t)\|\ | \ t\in
\mathbb R,\ \mathfrak{g}\in \mathbb F\}$.
\end{lemma}
\begin{proof} Let $v(t)$ be a solution of the equation
(\ref{eq2.4}) with the initial condition $v(0)=v_0$ defined on the
maximal interval $[0,h)$. Denote by $y(t):=\|v(t)\|^{2}$ then we
have
\begin{eqnarray}\label{eqDE_1}
& y'(t)=2\langle v'(t),v(t)\rangle =2\langle \nu A
v(t),v(t)\rangle +2\langle \Phi(v(t)),v(t)\angle +2\langle \mathfrak{g}(t),v(t)\rangle = \nonumber \\
& -2\nu \|D^{+}v(t)\|^{2}-2\lambda \|v(t)\|^{2}+2\langle
\tilde{F}(v(t)),v(t)\rangle +2\langle \mathfrak{g}(t),v(t)\rangle
\end{eqnarray}
for all $t\in [0,h)$.

Since
\[
|\langle \mathfrak{g}(t),v(t)\rangle| \le
\|\mathfrak{g}(t)\|\|v(t)\| \le \frac{1}{2}\lambda \|v(t)\|^2 +
\frac{1}{2\lambda}\|\mathfrak{g}(t)\|^2\,,
\]
using (\textbf{C3}) from (\ref{eqDE_1}) we get
\begin{equation}\label{eqDE2.1}
y'(t)=2\langle v'(t),v(t)\rangle \le -(\lambda +2\alpha)y(t)+
\frac{C^{2}}{\lambda}
\end{equation}
for every $t\in [0,h)$. By Gronwall's lemma from the inequality
(\ref{eqDE2.1}), taking into account that $y(0)=\|v(0)\|^{2}$, we
obtain
\begin{equation}\label{eqDE_03}
y(t)\le e^{-(\lambda
+2\alpha)t}\big{(}(\|v(0)\|^{2}-\frac{C^{2}}{\lambda (\lambda
+2\alpha)})+\frac{C^{2}}{\lambda (\lambda
+2\alpha)}\big{)}\nonumber
\end{equation}
and, consequently,
\begin{equation}\label{eqDE_004}
\|v(t)\|\le M \nonumber
\end{equation}
for all $0\le t\le T<h$, where
\begin{equation}\label{eqDE5}
M=M(T,\|v_0\|,C):=\Big{(}e^{-(\lambda
+2\alpha)T}\big{(}(\|v_0\|^{2}-\frac{C^{2}}{\lambda (\lambda
+2\alpha)})+\frac{C^{2}}{\lambda (\lambda
+2\alpha)}\big{)}\Big{)}^{1/2}.\nonumber
\end{equation}
Lemma is proved.
\end{proof}

\begin{remark}\label{remDC1} Lemma \ref{lBC_1} remains true if we
replace the Condition (\textbf{C3}) by the weaker condition:
$F(s)s\le 0$ for all $s\in \mathbb R$.
\end{remark}

\begin{theorem} Under the Conditions (\textbf{C1})-(\textbf{C3}) the following statements hold:
\begin{enumerate}
    \item for every $(v,\mathfrak{g})\in \ell_{2}\times \mathbb F$ there exists a unique
    solution $\varphi(t,v,\mathfrak{g})$ of the equation (\ref{eq2.4}) passing
    through the point $v$ at the initial moment $t=0$ and defined on
    the semi-axis $\mathbb R_{+}:=[0,+\infty)$; \item
    $\varphi(0,v,\mathfrak{g})=v$ for all $(v,\mathfrak{g})\in \ell_{2}\times \mathbb F$; \item
    $\varphi(t+\tau,v,\mathfrak{g})=\varphi(t,\varphi(\tau,v,\mathfrak{g}),\mathfrak{g}^{\tau})$ for
    all $t,\tau\in \mathbb R_{+}$, $v\in \ell_{2}$ and $\mathfrak{g}\in \mathbb F$;
     \item the mapping
    $\varphi :\mathbb R_{+}\times \ell_{2}\times \mathbb F\to \ell_{2}$
    ($(t,v,\mathfrak{g})\to \varphi(t,v,\mathfrak{g}))$ for all $(t,v,\mathfrak{g})\in \mathbb
    R_{+}\times \ell_{2}\times \mathbb F$ is continuous.
\end{enumerate}
\end{theorem}
\begin{proof}
The first statement of Theorem follows from Lemma \ref{lBC_1},
Theorem \ref{thLS1.S1} and Remark \ref{remSL1}.

The second and third statements are evident. The fourth statement
follows from Theorem \ref{thLS1.S1}.
\end{proof}

\begin{coro}\label{corH1}
Under the conditions of Theorem \ref{th1.1} the equation
(\ref{eq2.3}) (respectively, the family of equations
(\ref{eq2.4})) generates a cocycle $\langle
\ell_{2},\varphi,(\mathbb F,\mathbb R,\sigma)\rangle$ over the
shift dynamical system $(\mathbb F,\mathbb R,\sigma)$ with the
fibre $\ell_{2}$.
\end{coro}
\begin{proof} This statement directly follows from Theorem
\ref{th1.1} and Definition \ref{def1.0.19}.
\end{proof}

\begin{theorem}\label{thCGA2} \cite{CS_2026} Under the Conditions
(\textbf{C1})-(\textbf{C3}) the equation (\ref{eq2.3}) (the
cocycle $\varphi$ generated by the equation (\ref{eq2.3})) has a
compact global attractor $\{I_{\mathfrak{g}}|\ \mathfrak{g}\in
H(\mathfrak{f})\}$.
\end{theorem}

\section{Uniform Dissipativity of the Finite Dimensional
Approximation Systems}\label{Sec6}

It follows from standard existence and uniqueness theorems for
finite dimensional ODEs the Cauchy problem (\ref{eq2})-(\ref{eq3})
has a unique global solution $\varphi_{n}(t,u,g^{n})$. Then the
mapping $\varphi_{n}:\mathbb{R}_{+}\times \mathbb R^{2n+1}\times
H(\mathfrak{f}^{n})\to \mathbb R^{2n+1}$ is well defined and
satisfies the following conditions (see, e.g.
\cite{Bro79,Che_2015,Sel_1971}):
\begin{enumerate}
\item  $\varphi_{n}(0,v,\widetilde{\mathfrak{g}})=v$ for all $v\in
\mathbb R^{2n+1}$ and $\widetilde{\mathfrak{g}}\in
H(\mathfrak{f}^{n})$; \item
$\varphi_{n}(t,\varphi_{n}(\tau,v,\widetilde{\mathfrak{g}}),
\widetilde{\mathfrak{g}}^{\tau})=\varphi_{n}(t+\tau,v,\widetilde{\mathfrak{g}})$
for all $ v\in \mathbb R^{2n+1}$, $\widetilde{\mathfrak{g}}\in
H(\mathfrak{f}^{n})$ and $t,\tau \in \mathbb{R}_{+}$; \item  the
mapping $\varphi_{n}:\mathbb{R}_{+}\times \mathbb R^{2n+1}\times
H(g^{n})\to \mathbb R^{2n+1}$ is continuous.
\end{enumerate}

Denote by
$H(\mathfrak{f}^{n}):=\overline{\{\sigma(h,\mathfrak{f}^{n})|\
h\in \mathbb R\}}$ and $(H(\mathfrak{f}^{n}),\mathbb{R},\sigma)$
the shift dynamical system on $H(\mathfrak{f}^{n})$ induced from
$(C(\mathbb R\times\mathbb R^{2n+1},\mathbb R^{2n+1}),\mathbb
R,\sigma)$, i.e.,
$\sigma(\tau,\widetilde{\mathfrak{g}})=\widetilde{\mathfrak{g}}^\tau$
for all $\tau\in\mathbb R$ and $\widetilde{\mathfrak{g}}\in
H(\mathfrak{f}^{n})$.

Thus the equation (\ref{eq2}) generates a nonautonomous (cocycle)
semi-dynamical system  on $\langle \mathbb R^{2n+1},\varphi_{n},
(H(\mathfrak{f}^{n}),\mathbb R,\sigma)\rangle$ (or shortly
$\varphi_{n}$).

We have a sequence \(\{\varphi_{n}\}\) of the cocycles with the
phase space $\mathbb R^{2n+1}$ of \(\varphi_{n}\).

\begin{lemma}\label{lBC_01} Assume that the conditions (\textbf{C1})--(\textbf{C3}) holds,
the function $g\in C(\mathbb R,\ell_{2})$ is Lagrange stable and
for every $n\in \mathbb N$,
$\mathfrak{f}^{n}:=P_{n}(\mathfrak{f})$ and
$\widetilde{\mathfrak{g}} \in H(\mathfrak{f}^{n})$. Then, for all
$T > 0$ and solution $v(t)$ of the problem (\ref{eq2})-(\ref{eq3})
and $v(0)=v_0 \in \mathbb R^{2n+1}$ satisfies
$$
\|v(t)\| \le M \;,\quad \mbox{for all} \ 0 \le t \le T\,,
$$
where $M$ is a constant depending only on the data $(\lambda, C,
\|v_0\|)$ and $T$, where $C:=\sup\{\|\mathfrak{g}(t)\|\ | \ t\in
\mathbb R\}$.
\end{lemma}
\begin{proof} Let $v(t)$ be a solution of the equation
(\ref{eq2}) with the initial condition $v(0)=v_0$ defined on the
maximal interval $[0,h)$. Denote by $y(t):=\|v(t)\|^{2}$ then we
have
\begin{eqnarray}\label{eqDE1}
& y'(t)=2(v'(t),v(t))=2(\nu A_{n}
v(t),v(t))+2(f^{n}(v(t)),v(t))+2(\mathfrak{f}^{n}(t),v(t))= \nonumber \\
& -2\nu \|B_{n}^{+}v(t)\|^{2}-2\lambda
\|v(t)\|^{2}+2(\widetilde{f}^{n}(v(t)),v(t))+2(\mathfrak{f}^{n}(t),v(t))
\nonumber
\end{eqnarray}
for every $t\in [0,h)$.

Note that
\begin{equation}\label{eqDE1.1}
\|\mathfrak{f}^{n}(t)\|^{2}=\sum\limits_{|i|\le
n}|\mathfrak{f}_{i}(t)|^{2}\le \sum\limits_{i\in \mathbb
Z}|\mathfrak{f}_{i}(t)|^{2}\le
C^{2}:=\sup\{\|\mathfrak{g}(t)\|^{2}|\ (t,\mathfrak{g})\in \mathbb
R\times\mathbb F\}
\end{equation}
for every $n\in \mathbb N$.

Since
\[
|(\mathfrak{f}^{n}(t),v(t))| \le \|\mathfrak{f}^{n}(t)\|\|v(t)\|
\le \frac{1}{2}\lambda \|v(t)\|^2 +
\frac{1}{2\lambda}\|\mathfrak{f}^{n}(t)\|^2\,,
\]
using (\textbf{C3}) from (\ref{eqDE1.1}) we get
\begin{equation}\label{eqDE2}
y'(t)=2(v'(t),v(t)) \le -(\lambda +2\alpha)y(t)+
\frac{C^{2}}{\lambda}
\end{equation}
for every $t\in [0,h)$. By Gronwall's lemma from the inequality
(\ref{eqDE2}), taking into account that $y(0)=|v(0)|^{2}$, we
obtain
\begin{equation}\label{eqDE3}
y(t)\le e^{-(\lambda
+2\alpha)t}\big{(}(\|v(0)\|^{2}-\frac{C^{2}}{\lambda (\lambda
+2\alpha)})+\frac{C^{2}}{\lambda (\lambda
+2\alpha)}\big{)}\nonumber
\end{equation}
and, consequently,
\begin{equation}\label{eqDE_04}
\|v(t)\|\le M \nonumber
\end{equation}
for all $0\le t\le T<h$, where
\begin{equation}\label{eqDE_05}
M=M(T,\|v_0\|,C):=\Big{(}e^{-(\lambda
+2\alpha)T}\big{(}(\|v_0\|^{2}-\frac{C^{2}}{\lambda (\lambda
+2\alpha)})+\frac{C^{2}}{\lambda (\lambda
+2\alpha)}\big{)}\Big{)}^{1/2}.\nonumber
\end{equation}
Lemma is proved.
\end{proof}

\begin{remark}\label{remDC_1} Lemma \ref{lBC_01} remains true if we
replace the Condition (\textbf{C3}) by the weaker condition:
$F(s)s\le 0$ for all $s\in \mathbb R$.
\end{remark}

Denote by $\mathcal{O}_n := \{ v \in \mathbb{R}^{2n+1} \mid
\|v\|_{\mathbb{R}^{2n+1}} \le M \}$.

By Lemma \ref{lBC_01}, $\mathcal{O}_n$ is a bounded absorbing set
for the cocycle $\langle \mathbb{R}^{2n+1}, \varphi_n,
(H(\mathfrak{f}^{n}), \mathbb{R}, \sigma) \rangle$ (or shortly
$\varphi_{n}$ ) in the space $\mathbb{R}^{2n+1}$.

According to Theorem 5 \cite{Che_2022} (see also \cite[Ch.II,
Theorem 2.8.6]{Che_2024}), the cocycle $\varphi_n$ has a compact
global attractor
$I^{n}:=\{I^n_{\widetilde{\mathfrak{g}}}|\widetilde{\mathfrak{g}}\in
H(\mathfrak{f}^{n})\}\) and $I^n_{\widetilde{\mathfrak{g}}}\subset
\mathcal{O}_n$ for all $\widetilde{\mathfrak{g}} \in
H(\mathfrak{f}^{n})$.

\begin{lemma}\label{l2} Assume that
$$
F(s)s \le 0 \quad \forall s \in \mathbb{R}
$$
and $\mathfrak{f} \in C(\mathbb{R}, \ell_2)$ is Lagrange stable.
Then for each $\varepsilon > 0$, there exists $k(\varepsilon) > 0$
depending on the data $(\nu, \lambda, \widetilde{g})$ and
$\varepsilon$ but not on $n$ such that $\forall w \in I^n :=
\bigcup\{I^n_{\widetilde{\mathfrak{g}}}|\widetilde{\mathfrak{g}}\in
H(\mathfrak{f}^{n})\}$
$$
\sum_{k(\varepsilon) \le |i| \le n} |w_i|^2 \le \varepsilon
$$
\end{lemma}
\begin{proof} Note that the space $\mathbb R^{2n+1}$ is embedded
in the space $\ell_{2}$ as follow:
$$
u\to \mathcal{I}_{n}(u):=(\ldots,0,u_{-n},u_{-n+1},\ldots,
u_{-1},u_{0},u_{1}.\dots,u_{n-1},u_{n},0,\ldots).
$$
It is evident that this embedding (the mapping $\mathcal{I}_{n}$)
is isometric and completely continuous. We will prove Lemma
\ref{l2} using the slight modification of the ideas and methods
elaborated in the work \cite{BLW_2001} (see also
\cite[Ch.2.3]{HK_2023}). Consider a smooth function \( \xi :
\mathbb{R}^+ \to [0, 1] \) satisfying
\[
\xi(s) =
\begin{cases}
0, & 0 \le s \le 1, \\
\in [0, 1], & 1 \le s \le 2, \\
1, & s \ge 2
\end{cases}
\]
and note that there exists a constant $C_0$ such that $|\xi'(s)|
\le C_0$ for all \( s \ge 0 \). Then for a fixed \( k \in
\mathbb{N} \) (its value will be specified later), define
\[
\xi_k(s) = \xi \left( \frac{s}{k} \right) \quad \text{for all}
\quad s \in \mathbb{R_+}.
\]

Given \( u \in C^{1}(\mathbb R_{+},\mathbb  R^{2n+1}) \), define
\( v \in C^{1}(\mathbb R_{+}, \mathbb R^{2n+1}) \) componentwise
as
\[
v_i (t):= \xi_k( |i| ) u_i(t) \quad \text{for} \quad i \in
\mathbb{Z}\ \ (\forall \ t\in \mathbb R_{+}).
\]

Note that
\begin{equation}\label{eqNT1}
\langle u(t),v(t)\rangle =\sum\limits_{i\in \mathbb
Z}\xi_{k}(|i|)|u_{i}(t)|^{2} \nonumber
\end{equation}
and
\begin{equation}\label{eqNT2}
\frac{d\langle u(t),v(t)\rangle}{dt}=2\langle
\frac{du(t)}{dt},v(t)\rangle \nonumber
\end{equation}
for every $t\in \mathbb R_{+}$.

Taking the inner product of the equation (\ref{eq2}) with \(
\mathbf{v}(t) \) gives
\[
\frac{d}{dt} \langle u(t),v(t) \rangle + \nu \langle B_{n}u(t),
B_{n}v(t) \rangle = \langle \Phi_{n}u(t)), v(t) \rangle + \langle
\mathfrak{f}^{n}(t), v(t) \rangle,
\]

that is
\begin{equation}\label{20}
\frac{d}{dt} \sum_{|i| \le n} \xi_k( |i| ) |u_i|^2 + 2 \nu \langle
B_{n} u, B_{n} v \rangle = 2 \sum_{|i| \le n} \xi_k( |i| ) u_i
f(u_i) +
\end{equation}
$$
2 \sum_{|i| \le n} \xi_k( |i| ) \mathfrak{f}^{n}_{i}(t)u_i .
$$

Each term in the equation (\ref{20}) will now be estimated. First,
\[
\left\langle B_{n}u,  B_{n}v \right\rangle = \sum_{|i| \le n}
(u_{i+1} - u_i)(v_{i+1} - v_i)
\]
\[
= \sum_{|i| \le n} (u_{i+1} - u_i) \left[ \left( \xi_k(|i+1|) -
\xi_k(|i|) \right) u_{i+1} + \xi_k(|i|)(u_{i+1} - u_i) \right]
\]
\[
= \sum_{|i| \le n} \left( \xi_k(|i+1|) - \xi_k(|i|) \right)
(u_{i+1} - u_i) u_{i+1} + \sum_{|i| \le n} \xi_k(|i|) (u_{i+1} -
u_i)^2
\]
\[
\ge \sum_{|i| \le n} \left( \xi_k(|i+1|) - \xi_k(|i|) \right)
(u_{i+1} - u_i) u_{i+1}.
\]
Since
\[
\left| \sum_{|i| \le n} \left( \xi_k(|i+1|) - \xi_k(|i|) \right)
(u_{i+1} - u_i) u_{i+1} \right| \le \sum_{|i| \le n} \frac{1}{k}
|\xi'(s_i)| \cdot |u_{i+1} - u_i| \cdot |u_{i+1}|,
\]
for some \( s_i \) between \( |i| \) and \( |i+1| \), and
\[
\sum_{|i| \le n} |\xi'(s_i)| \, |u_{i+1} - u_i| \, |u_{i+1}| \le
C_0 \sum_{|i| \le n} \left( |u_{i+1}|^2 + |u_i||u_{i+1}| \right)
\le 4 C_0 \left\| \mathbf{u} \right\|^2.
\]

Then it follows that for all \( \mathbf{u} \in Q \) and \(
\mathbf{v} \in \ell^2 \) defined componentwise as \( v_i :=
\xi_k(|i|) u_i \), for \( |i| \le N \),
\begin{equation}\label{In20}
\left\langle B_{n}u, B_{n}v \right\rangle \ge - \frac{4 C_0
\|Q\|^2}{k}.
\end{equation}
where \( \|Q\| := \sup_{\mathbf{u} \in Q} \left\| \mathbf{u}
\right\| \). On the other hand, by Condition (\textbf{C3}),
\begin{equation}\label{eqUD1}
2 \sum_{|i| \le n} \xi_k(|i|) u_i f(u_i) \le -2 \alpha \sum_{|i|
\le n} \xi_k(|i|) |u_i|^2
\end{equation}

and by Young's inequality
\begin{equation}\label{eqUD2}
2 \sum_{|i| \le n} \xi_k(|i|) g_i u_i \le \alpha \sum_{|i| \le n}
\xi_k(|i|) |u_i|^2 + \frac{1}{\alpha} \sum_{|i| \le n} \xi_k(|i|)
|\mathfrak{f}^{n}_i(t)|^2.
\end{equation}

From (\ref{eqUD1}) and (\ref{eqUD2}) we obtain
\begin{equation}\label{22}
2 \sum_{|i| \le n} \xi_k(|i|) u_i f(u_i) + 2 \sum_{|i| \le n}
\xi_k(|i|) \mathfrak{f}^{n}_i(t) u_i \le -\alpha \sum_{|i| \le n}
\xi_k(|i|) |u_i|^2 +
\end{equation}
$$
\frac{1}{\alpha} \sum_{n \ge |i| \ge k} |\mathfrak{f}^{n}_i(t)|^2
.
$$

Using the estimates (\ref{In20}) and (\ref{22}) in the equation
(\ref{20}) gives
\begin{equation}\label{EqCDN1}
\frac{d}{dt} \sum_{|i| \le n} \xi_k(|i|) |u_i|^2 + \alpha
\sum_{|i| \le n} \xi_k(|i|) |u_i|^2 \le \nu \frac{4 C_0 \| Q
\|^2}{k} + \frac{1}{\alpha} \sum_{n \ge |i| \ge k}
|\mathfrak{f}^{n}_i(t)|^2
\end{equation}

Since $H(K)$ ($K=\{\mathfrak{f}^{n}\}_{n\in \mathbb Z_{+}}$) is a
compact subset in the space $C(\mathbb R,\ell_{2})$ then by Lemma
\ref{lAPF02} there exists a compact subset $Q$ from $\ell_{2}$
such that $\mathfrak{g}(\mathbb R)\subset Q$ for all
$\mathfrak{g}\in H(K)$. In particular for every $\varepsilon >0$
there exists a natural number $k(\varepsilon) \le n$ such that
\begin{equation}\label{eqDC2}
\sum\limits_{n\ge |i|\ge k(\varepsilon)}|v_i|^{2}<\varepsilon
\end{equation}
for all $v\in \overline{\mathfrak{f}^{n}(\mathbb R)}$. Note that
$H(\mathfrak{f}^{n})\subseteq H(K)$ and, consequently,
\begin{equation}\label{eqDC3}
\mathfrak{f}^{n}(t)\in \overline{\mathfrak{f}^{n}(\mathbb
R)}\subseteq Q.
\end{equation}
From (\ref{eqDC2}) and (\ref{eqDC3}) we obtain
\begin{equation}\label{eqDC4}
\sum\limits_{n \ge |i|\ge
k(\varepsilon)}|\mathfrak{f}^{n}_i(t)|^{2}<\varepsilon \nonumber
\end{equation}
for all $n\in \mathbb Z_{+}$ and $t\in \mathbb R$.

Note that $\mathfrak{f}^{n}(t)\in \overline{\mathfrak{g}(\mathbb
R)}\subseteq \overline{\mathfrak{f}^{n}(\mathbb R)}$ for all $t\in
\mathbb R$ and, consequently, for every \( \varepsilon
> 0 \), there exists $k_{\varepsilon}\le n$ such that
\[
\nu \frac{4C_0 \|Q\|^2}{k} + \frac{1}{\alpha} \sum_{n \ge|i| \ge
k} |\mathfrak{f}^{n}_{i}(t)|^2 \le \varepsilon, \quad n \ge k \ge
k(\varepsilon)
\]
for every $n\in \mathbb Z_{+}$ and $t\in \mathbb R$.

The inequality (\ref{EqCDN1}) along with the relation above give
\[
\frac{d}{dt} \sum_{|i| \le n} \xi_k(|i|) |u_i|^2 + \alpha
\sum_{|i| \le n} \xi_k(|i|) |u_i|^2 \le \varepsilon
\]
for all $n\in \mathbb Z_{+}$ and $t\in \mathbb R$.

Then, Gronwall's lemma implies that
\[
\sum_{|i| \le n} \xi_k(|i|)
|\varphi_{n}(t,u_0,\mathfrak{f}^{n})_i|^2 \le e^{-\alpha t}
\sum_{|i| \le N} \xi_k(|i|) |v_{i}|^2 + \frac{\varepsilon}{\alpha}
\le e^{-\alpha t} \|v\|^2 + \frac{\varepsilon}{\alpha}.
\]

Hence for every \( v \in Q \),
\[
\sum_{|i| \le n} \xi_k(|i|)
|\varphi(t,u_0,\mathfrak{f}^{n})_{i}|^2 \le e^{-\alpha t} \| Q
\|^2 + \frac{\varepsilon}{\alpha}.
\]
and therefore
\begin{equation}\label{eqT1}
\sum_{|i| \le n} \xi_k(|i|) \left| \varphi(t,
u_{0},\mathfrak{f}^{n})_{i} \right|^2 \le
\frac{2\varepsilon}{\alpha}, \quad \text{for } t \ge
T(\varepsilon) := \frac{1}{\alpha} \ln \frac{\alpha \| Q
\|^2}{\varepsilon}.
\end{equation}
This means that the cocycle $\langle
\ell_{2},\varphi_{n},(H(\mathfrak{f}^{n}),\mathbb
R,\sigma)\rangle$ is asymptotic tails on the absorbing set $Q$.
Lemma is proved.
\end{proof}

\begin{coro}\label{cor2} Under the conditions of Lemma \ref{l2}
the set
\begin{equation}\label{eqCY1}
\bigcup_{n\in \mathbb Z_{+}}\{I^{n}_{\widetilde{\mathfrak{g}}}|\
\widetilde{\mathfrak{g}}\in H(\mathfrak{f}^{n})\}
\end{equation}
is precompact in the space $\ell_{2}$.
\end{coro}
\begin{proof}
Let
$$
w\in I:=\bigcup\{I^{n}|\ n\in \mathbb Z_{+}\},
$$
where
$$
I^{n}:=\bigcup\{I^{n}_{\mathfrak{g}}|\ \mathfrak{g}\in
H(\mathfrak{f}^{n})\},
$$
then there exist $n\in \mathbb Z_{+}$ and $\mathfrak{g}\in
H(\mathfrak{f}^{n})$ such that $w\in I^{n}_{\mathfrak{g}}$. Since
$$
I^{n}_{\mathfrak{g}}=\varphi(\tau,I^{n}_{\mathfrak{g}},\sigma(-\tau,\mathfrak{g}))
$$
for every $\tau \in \mathbb R_{+}$ then there exists a point $v\in
I^{n}_{\mathfrak{g}}$ such that
$$
w=\varphi(T(\varepsilon),v,\sigma(-T(\varepsilon),\mathfrak{g})),
$$
where $T(\varepsilon)$ is the positive number figuring in
(\ref{eqT1}). Thus we obtain
\begin{equation}\label{eqT2}
\|v\|^{2}=\sum\limits_{k_{\varepsilon}\le |i|\le n}|w_{i}|^{2}=
\sum\limits_{k_{\varepsilon}\le |i|\le
n}|\varphi(T(\varepsilon),w,\sigma(-T(\varepsilon),\mathfrak{g}))|^{2}.
\end{equation}
From (\ref{eqT1}) and (\ref{eqT2}) one has
\begin{equation}\label{eqT3}
\sum\limits_{k_{\varepsilon}\le |i|\le n}|w_{i}|^{2}<\varepsilon .
\end{equation}
According to Theorem 5.25 \cite[Ch.V,p.167]{LS_1974} the set $I$
is precompact in $\ell_{2}$. The Corollary is proved.
\end{proof}

\section{Upper Semi-continuous Convergence of the
Finite-dimensional Attractors}\label{Sec7}

Below we give a Lemma which play an important role in the proof of
the upper semi-continuity of the global attractor.

\begin{lemma}\label{lUSC1} Assume that
\begin{equation}\label{eqFUSC1}
F(s)s\le 0
\end{equation}
for all $s\in \mathbb R$, $\mathfrak{f}\in C_{b}(\mathbb
R,\ell_{2})$ and $v_{0,n}\in I^{n}_{\mathfrak{f}^{n}}$. Then there
exists a subsequence $\{v_{0,n_{k}}\}$ of $\{v_{0,n}\}$ and $u\in
I_{\mathfrak{f}}$ such that $v_{0,n_{k}}$ converges to $u$ in
$\ell_{2}$.
\end{lemma}
\begin{proof} We will prove this statement by slight modifications
of the reasoning from the work \cite{BLW_2001} (see the proof of
Lemma 4.3).

Let $\nu_{n}(t)=\varphi_{n}(t,v_{0,n},\mathfrak{f}^{n})$ be the
solution of (\ref{eq2}) with the initial condition (\ref{eq3}).
Since $v_{0,n}\in I^{n}_{\mathfrak{f}^{n}}$ then $\nu_{n}(t)\in
I^{n}_{\sigma(t,\mathfrak{f}^{n})}$ for every $t\in \mathbb R$.
Then it follows from (\ref{eq2}) that
\begin{equation}\label{eqUSC1}
\|\nu_{n}(t)\|\le M
\end{equation}
for all $t\in \mathbb R$ and $n\in \mathbb N$. By (\ref{eq2}) we
get
\begin{equation}\label{eqUSC2}
\|\nu_{n}'(t)\|\le \nu \|A_n \nu_{n}(t)\| =\lambda \|\nu_{n}(t)\|
\|f(\nu_{n}(t))\|+\|\mathfrak{f}(t)\|.
\end{equation}
Note that
\begin{equation}\label{eqUSC3}
\nu \|A_{n}\nu_{n}(t)\|\le 4\nu \|\nu_{n}(t)\|\le 4\nu M
\end{equation}
and
\begin{equation}\label{eqUSC4}
\|f(\nu_{n}(t))\|=\|f(\nu_{n}(t))-f(0)\|\le C\|\nu_{n}(t)\|\le CM
\end{equation}
for all $t\in \mathbb R$ and $n\in \mathbb N.$

From (\ref{eqUSC2})-(\ref{eqUSC4}) it follows that
\begin{equation}\label{eqUSC5}
\|\nu'_{n}(t)\|\le C
\end{equation}
and, consequently,
 \begin{equation}\label{eqUSC6}
 \|\nu_{n}(t)-\nu_{n}(s)\|\le C|t-s|
 \end{equation}
 for all $t\in \mathbb R$ and $n\in \mathbb N$.

We shell prove that the sequence $\{\nu_{n}\}_{n\in \mathbb N}$ is
precompact in $C(\mathbb R,\ell_{2})$. By Lemma \ref{lAPF02} to
prove that $\{\nu_{n}\}_{n\in \mathbb N}$ is precompact it is
sufficient to show that the following two conditions are
fulfilled:
\begin{enumerate}
\item there exists a compact subset $Q$ from $\ell_{2}$ such that
$\nu_{n}(\mathbb R)\subseteq Q$ for every $n\in \mathbb N$; \item
the functional sequence $\{\nu_{n}\}_{n\in \mathbb N}$ is
uniformly equi-continuous.
\end{enumerate}
To this end we note that from (\ref{eqUSC6})
 it follows that the sequence $\{\nu_{n}\}_{n\in \mathbb N}$ is
 uniformly equi-continuous.

 Finally, we note that by Corollary \ref{cor2} the set
 \begin{equation}\label{eqUSC7}
 Q:=\overline{\bigcup\limits_{n\in \mathbb N}I^{n}} \nonumber
 \end{equation}
is compact in $\ell_{2}$, where
$I^{n}:=\{I^{n}_{\widetilde{\mathfrak{g}}}|\
\widetilde{\mathfrak{g}}\in H(\mathfrak{f}^{n})\}$. Since
$\nu_{n}(t)\in I^{n}_{\sigma(t,\mathfrak{f}^{n})}\subseteq Q$ for
all $(n,t)\in \mathbb N \times \mathbb R$, then $\nu_{n}(\mathbb
R)\subseteq Q$ fr every $n\in \mathbb N$. By Lemma \ref{l2} the
functional sequence $\{\nu_{n}\}_{n\in \mathbb N}$ is precompact
in $C(\mathbb R,\ell_{2})$ and, consequently, there exists a
subsequence $\{\nu_{n_{k}}\}\subset C(\mathbb R,l_{2})$ and $u\in
C(\mathbb R,l_{2})$ such that
\begin{equation}\label{eqUSC7.1}
\nu_{n_{k}}\to u \ \ (\mbox{in the space}\ C(\mathbb R,l_{2}))
\end{equation}
as $k\to \infty$.


By (\ref{eqUSC1}) we get
\begin{equation}\label{eqLUSC1}
\|u(t)\|\le C
\end{equation}
for all $t\in \mathbb R$. We now prove $u(0)\in I_{\mathfrak{f}}$.
To this end, we show that $u$ is a solution to (\ref{eq2}) and
(\ref{eq3}). For simplicity, we denote by $\{\nu_{n}\}$ the
sequence $\{\nu_{n_{k}}\}$ in (\ref{eqUSC7.1}). It follows from
(\ref{eqUSC5}) that
\begin{equation}\label{eqLUSC2}
\dot{\nu}_{n}\to \dot{u}\ \mbox{in}\ L^{\infty}(\mathbb
R,\ell_{2})\ \mbox{weak star}.
\end{equation}
For fixe $i\in \mathbb Z$, let $n>|i|$. Since
$\nu_{n}(t)=\varphi_{n}(t,v_{0,n},\mathfrak f^{n})$ is the
solution to (\ref{eq2}) and (\ref{eq3}), we have
\begin{equation}\label{eqLUSC3}
\dot{\nu}_{n,i} =-\nu (-\nu_{n,i-1}+2\nu_{n,i} -\nu_{n.i+1}-
-\lambda \nu_{n,i}-F(\nu_{n,i})+\mathfrak{f}_{i}(t),\  t\in
\mathbb R .
\end{equation}
Then for each $\psi \in C_{0}^{\infty}(\mathbb R)$, we get
$$
 \int\limits_{\mathbb R}\dot{\nu}_{n,i}(t)\psi(t)dt=-\nu
\int\limits_{\mathbb
R}(-\nu_{n,i-1}(t)+2\nu_{n,i}(t)-\nu_{n,i+1}(t))\psi(t)dt-
$$
\begin{equation}\label{eqLUSC4}
\lambda\int\limits_{\mathbb
R}\nu_{n,i}(t)\psi(t)dt-\int\limits_{\mathbb
R}F(\nu_{n,i}(t))\psi(t)dt +\int\limits_{\mathbb
R}\mathfrak{f}_{i}(t)\psi(t)dt ,
\end{equation}
where by $C_{0}^{\infty}(\mathbb R)$ the family of all functions
$C^{\infty}(\mathbb R)$ with the compact support is denoted.

Note that
$$
\Big{|} \int\limits_{\mathbb R}F(\nu_{n,i}(t))\psi(t)dt
-\int\limits_{\mathbb R}F(u_{i}(t))\psi(t)dt \Big{|} \le
\int\limits_{\mathbb R}|F(\nu_{n,i}(t))-F(u_{i}(t))||\psi(t)|dt
\le
$$
\begin{equation}\label{eqLUSC5}
\sup\limits_{t\in \mathbb R}|F'(\xi_{i})|\sup\limits_{t\in \mathbb
R}|v^{n}_{i}(t)-u_{i}(t)|\int\limits_{\mathbb R}|\psi(t)|dt \to 0
\end{equation}
because according to (\ref{eqUSC7.1}) we have $|\xi_{i}(t)|\le
\|v^{n}(t)\|+\|u(t)\|\le C$ for all $t\in \mathbb R$.

By (\ref{eqUSC7.1}), (\ref{eqLUSC2}) and (\ref{eqLUSC4}) we find
that $u$ satisfies
\begin{equation}\label{eqLUSC6}
\dot{u}_{i} = -\nu (-u_{i-1}+2u_{i} -u_{i+1}-\lambda u_{i}
-F(u_{i})+\mathfrak{f}_{i}(t)\nonumber
\end{equation}
for all $t \in \mathbb R$ and $i\in \mathbb Z$.

This together with (\ref{eqLUSC1}) shows that $u$ is a solution of
the equation (\ref{eq2.3}) which is bounded and defined on the
whole of $\mathbb R$. Hence, $u(t)\in I_{\sigma(t,\mathfrak{f})}$
for all $t\in \mathbb R$, i.e., $u(0)\in I_{\mathfrak{f}}$. By
(\ref{eqUSC7.1}) we get $\nu_{n_{i}}(0)\to u(0)\in
I_{\mathfrak{f}}$ which concludes the proof of Lemma.
\end{proof}

Finally, we state the main result of this section.

\begin{theorem}\label{thUSC1} Assume that (5) holds and $\mathfrak{f}\in C(\mathbb R,\ell_{2})$ is Lagrange stable.
Then we have
\begin{equation}\label{eqLUSC7}
\lim\limits_{n\to
\infty}\beta(I^{n}_{\mathfrak{f}^{n}},I_{\mathfrak{f}})=0.\nonumber
\end{equation}
\end{theorem}
\begin{proof} We argue by contradiction. If the conclusion is not true,
then there exist a sequence $u^{n_{k}}\in
I^{n_{k}}_{\mathfrak{f}^{n_k}}$ and $\delta
>0$ such that
\begin{equation}\label{eqLUSC8}
\beta(u^{n_{k}},I_{\mathfrak{f}})\ge \delta .
\end{equation}
On the other hand, by Lemma \ref{lUSC1}, there exists a
subsequence $u^{n_{k_{m}}}$ of $u^{n_{k}}$ such that
\begin{equation}\label{eqLUSC9}
\beta(u^{n_{k_{m}}},I_{\mathfrak{f}})\to 0 \nonumber
\end{equation}
which contradicts (\ref{eqLUSC8}). The theorem is completely
proved.
\end{proof}

\section{Funding}

This research was supported by the State Programs of the Republic
of Moldova "Remotely Almost Periodic Solutions of Differential
Equations (25.80012.5007.77SE)" and partially was supported by the
Institutional Research Program 011303 "SATGED", Moldova State
University.

\section{Data availability}

No data was used for the research described in the article.

\section{Conflict of Interest}

The authors declare that they  have not conflict of interest.

\medskip
\textbf{ORCID (D. Cheban):} https://orcid.org/0000-0002-2309-3823

\textbf{ORCID (A. Sultan):} https://orcid.org/0009-0003-9785-9291

\end{document}